\documentclass[11pt,a4paper]{article}
\usepackage{float}
\usepackage{makeidx}
\usepackage{layout}
\usepackage{array}
\usepackage{amssymb}
\usepackage{amsmath, longtable}
\usepackage{mathtools}
\usepackage{bbm}
\usepackage{tabularx} 
\usepackage{caption}
\usepackage[utf8]{inputenc}
\usepackage{fancyhdr}
\usepackage{amsthm}
\usepackage{comment}
\usepackage[linesnumbered,boxed,ruled,vlined,algo2e]{algorithm2e}
\usepackage{graphicx}
\usepackage{scalerel}
\usepackage{authblk}
\usepackage[caption=false]{subfig}
\usepackage{array}
\newcolumntype{H}{>{\setbox0=\hbox\bgroup}c<{\egroup}@{}}

\usepackage[numbers,comma,sort&compress]{natbib}
\usepackage[margin=1in]{geometry}

\usepackage{hyperref}
\hypersetup{
    colorlinks=true,
    linkcolor=blue,
    citecolor=blue,
    filecolor=magenta,      
    urlcolor=blue,
}

\DeclarePairedDelimiterX{\inp}[2]{\langle}{\rangle}{#1, #2}

\makeatletter
\def\namedlabel#1#2{\begingroup
    #2%
    \def\@currentlabel{#2}%
    \phantomsection\label{#1}\endgroup
}
\makeatother

\newtheorem{theorem}{Theorem}
\newtheorem{prop}{Proposition}
\newtheorem{lemma}{Lemma}
\theoremstyle{remark}

\theoremstyle{definition}
\newtheorem{definition}{Definition}

\makeindex
\numberwithin{equation}{section}
\numberwithin{lemma}{section}
\numberwithin{theorem}{section}
\numberwithin{definition}{section}
\numberwithin{prop}{section}
\numberwithin{remark}{section}

\usepackage[pagewise]{lineno}
\usepackage{longtable}

\mathtoolsset{showonlyrefs}
\usepackage{etoolbox}
\usepackage{textcase}
\makeatletter
\patchcmd{\@sect}{\uppercase}{\MakeTextUppercase}{}{}
\patchcmd{\@sect}{\uppercase}{\MakeTextUppercase}{}{}
\makeatother

\title{Memory-Efficient Approximation Algorithms for \textsc{Max-k-Cut} and Correlation Clustering}

\author[1,2,3]{Nimita Shinde}
\author[2]{Vishnu Narayanan}
\author[3]{James Saunderson}
\affil[1]{IITB-Monash Research Academy}
\affil[2]{Industrial Engineering and Operations Research, IIT Bombay}
\affil[3]{Electrical and Computer Systems Engineering, Monash University}

\date{}

\begin{document}

\maketitle

\begin{abstract}
\textsc{Max-k-Cut} and correlation clustering are fundamental graph partitioning problems. For a graph $G=(V,E)$ with $n$ vertices, the methods with the best approximation guarantees for \textsc{Max-k-Cut} and the \textsc{Max-Agree} variant of correlation clustering involve solving SDPs with $\mathcal{O}(n^2)$ constraints and variables. Large-scale instances of SDPs, thus, present a memory bottleneck. In this paper, we develop simple polynomial-time Gaussian sampling-based algorithms for these two problems that use $\mathcal{O}(n+|E|)$ memory and nearly achieve the best existing approximation guarantees. For dense graphs arriving in a stream, we eliminate the dependence on $|E|$ in the storage complexity at the cost of a slightly worse approximation ratio by combining our approach with sparsification.
\end{abstract}

\section{Introduction}

Semidefinite programs (SDPs) arise naturally as a relaxation of a variety of problems such as $k$-means clustering~\cite{awasthi2015relax}, correlation clustering~\cite{bansal2004correlation} and \textsc{Max-k-Cut}~\cite{frieze1997improved}. In each case, the decision variable is an $n\times n$ matrix and there are $d=\Omega(n^2)$ constraints. While reducing the memory bottleneck for large-scale SDPs has been studied quite extensively in literature~\cite{burer, journee2010low, ding2019optimal,yurtsever2017sketchy}, all these methods use memory that scales linearly with the number of constraints and also depends on either the rank of the optimal solution or an approximation parameter. A recent Gaussian-sampling based technique to generate a near-optimal, near-feasible solution to SDPs with smooth objective function involves replacing the decision variable $X$ with a zero-mean random vector whose covariance is $X$~\cite{shinde2021memory}.
This method uses at most $\mathcal{O}(n+d)$ memory, independent of the rank of the optimal solution. However, for SDPs with $d=\Omega(n^2)$ constraints, these algorithms still use $\Omega(n^2)$ memory and provide no advantage in storage reduction. In this paper, we show how to adapt the Gaussian sampling-based approach of~\cite{shinde2021memory} to generate an approximate solution with provable approximation guarantees to \textsc{Max-k-Cut}, and to the \textsc{Max-Agree} variant of correlation clustering on a graph $G=(V,E)$ with arbitrary edge weights using only $\mathcal{O}(|V|+|E|)$ memory.

\subsection{\textsc{Max-k-Cut}} \textsc{Max-k-Cut} is the problem of partitioning the vertices of a weighted undirected graph $G=(V,E)$ into $k$ distinct parts, such that the total weight of the edges across the parts is maximized. If $w_{ij}$ is the edge weight corresponding to edge $(i,j)\in E$, then the cut value of a partition is
$\texttt{CUT} = \sum_{i\ \textup{and}\ j\ \textup{are in different partitions}} w_{ij}$.
Consider the standard SDP relaxation of \textsc{Max-k-Cut}
\begin{equation}\label{prob:maxkcutP}\tag{k-Cut-SDP}
\max_{X\succeq 0}\quad\langle C, X\rangle\quad\textup{subject to}\quad \begin{cases}&\textup{diag}(X) =\mathbbm{1}\\
&X_{ij} \geq -\frac{1}{k-1}\quad i\neq j,\end{cases}
\end{equation}
where $C= \frac{k-1}{2k}L_G$ is a scaled Laplacian. \citet{frieze1997improved} developed a randomized rounding scheme that takes an optimal solution $X^{\star}$ of~\eqref{prob:maxkcutP} and produces a random partitioning satisfying
\begin{equation}\label{eqn:maxcutbound}
\mathbb{E}[\texttt{CUT}] = \sum_{ij\in E, i<j}w_{ij}\textup{Pr}(i\ \textup{and}\ j\ \textup{are in different partitions}) \geq \alpha_k \langle C, X^{\star}\rangle\geq  \alpha_k \textup{opt}_k^G,
\end{equation}
where $\textup{opt}_k^G$ is the optimal $k$-cut value and $\alpha_k = \min_{-1/(k-1)\leq \rho \leq 1} \frac{kp(\rho)}{(k-1)(1-\rho)}$, where $p(\rho)$ is the probability that $i$ and $j$ are in different partitions given that $X_{ij} = \rho$. The rounding scheme proposed in~\cite{frieze1997improved}, referred to as the FJ rounding scheme in the rest of the paper, generates $k$ i.i.d. samples, $z_1,\dotsc, z_k\sim \mathcal{N}(0,X^{\star})$ and assigns vertex $i$ to part $p$, if $[z_p]_i \geq [z_l]_i$ for all $l = 1,\dotsc, k$.

\begin{comment}
\begin{definition}\label{def:kcutrounding}
Take $k$ i.i.d. Gaussian samples $z_1,\dotsc, z_k$ with covariance equal to an optimal solution of~\eqref{prob:maxkcutP}. The vertex $i$ is then assigned to partition $p$, if $[z_p]_i \geq [z_l]_i$ for $l \in \{1,\dotsc, k\}$.
\end{definition}
\end{comment}

\subsection{Correlation clustering}\label{sec:CC-Intro}
In correlation clustering, we are given a set of $|V|$ vertices together with the information indicating whether pairs of vertices are similar or dissimilar, modeled by the edges in the sets $E^+$ and $E^-$ respectively. The \textsc{Max-Agree} variant of correlation clustering seeks to maximize \[\mathcal{C} = \sum_{ij \in E^-}w_{ij}^-\mathbbm{1}_{[i,j \textup{ in different clusters}]}+\sum_{ij \in E^+}w_{ij}^+\mathbbm{1}_{[i,j \textup{ in the same cluster}]}.\]
Define $G^+ = (V,E^+)$ and $G^- = (V,E^-)$. A natural SDP relaxation of \textsc{Max-Agree}~\cite{bansal2004correlation} is
\begin{equation}\label{prob:CorrCluster}\tag{MA-SDP}
\max_{X\succeq 0}\quad \langle C, X\rangle\quad\textup{subject to}\quad \begin{cases}
&\textup{diag}(X) = \mathbbm{1}\\
 &X_{ij} \geq 0\quad i\neq j,
\end{cases}
\end{equation}
where $C = L_{G^-} + W^+$, $L_{G^-}$ is the Laplacian of the graph $G^-$ and $W^+$ is the weighted adjacency matrix of the graph $G^+$. \citet{charikar2005clustering} (see also~\citet{swamy2004correlation}) propose a rounding scheme that takes an optimal solution $X^{\star}_G$ of~\eqref{prob:CorrCluster} and produces a random clustering $\mathcal{C}$ satisfying
\begin{equation}\label{eqn:CC-Swamy}
\mathbb{E}[\mathcal{C}] \geq 0.766 \langle C, X^{\star}_G\rangle \geq 0.766 \textup{opt}^G_{CC},
\end{equation}
where $\textup{opt}_{CC}^G$ is the optimal clustering value. The rounding scheme proposed in~\cite{charikar2005clustering}, referred to as the CGW rounding scheme in the rest of the paper, generates either $k=2$ or $k=3$ i.i.d. zero-mean Gaussian samples with covariance $X^{\star}_G$ and uses them to define $2^k$ clusters.

\begin{comment}
$\sum_{ij\in E^-}w^-_{ij}(1-\langle x_i,x_j\rangle) +\sum_{ij\in E^+}w^+_{ij}\langle x_i,x_j\rangle$
\end{comment}

\subsection{Contributions}
We now summarize key contributions of the paper.

\textbf{Gaussian sampling for \textsc{Max-k-Cut}.} Applying Gaussian sampling-based Frank-Wolfe given in~\cite{shinde2021memory} directly to~\eqref{prob:maxkcutP} uses $n^2$ memory. We, however, show how to extend the approach from~\cite{shinde2021memory} to \textsc{Max-k-Cut} by proposing an alternate SDP relaxation for the problem, and combining it with the FJ rounding scheme to generate a $k$-cut with nearly the same approximation guarantees as stated in~\eqref{eqn:maxcutbound} (see Proposition~\ref{prop:MkCSummary}) using $\mathcal{O}(n+|E|)$ memory. A key highlight of our approach is that while the approximation ratio remains close to the state-of-the-art result in~\eqref{eqn:maxcutbound}, reducing it by a factor of $1-5\epsilon$ for $\epsilon \in (0,1/5)$, the memory used is independent of $\epsilon$. We summarize our result as follows.
\begin{prop}\label{prop:MkCSummary}
For $\epsilon \in (0,1/5)$, our $\mathcal{O}\left(\frac{n^{2.5}|E|^{1.25}}{\epsilon^{2.5}}\log(n/\epsilon)\log(|E|)\right)$-time method outlined in Section~\ref{sec:MkC} uses $\mathcal{O}(n+|E|)$ memory and generates a $k$-cut for the graph $G= (V,E)$ whose expected value satisfies $\mathbb{E}[\texttt{CUT}]\geq \alpha_k (1-5\epsilon)\textup{opt}_k^G$, where $\textup{opt}_k^G$ is the optimal $k$-cut value.
\end{prop}

\textbf{Gaussian sampling for \textsc{Max-Agree}.} The structure of~\eqref{prob:CorrCluster} is similar to that of~\eqref{prob:maxkcutP}, however, the cost matrix in~\eqref{prob:CorrCluster} is no longer PSD or diagonally dominant, a property that plays an important role in our analysis in the case of \textsc{Max-k-Cut}. Despite this, we show how to generate a $(1-7\epsilon)0.766$-optimal clustering using $\mathcal{O}(n+|E|)$ memory. Our approach makes a small sacrifice in the approximation ratio (as compared to~\eqref{eqn:CC-Swamy}), however, the memory used remains independent of $\epsilon$.
\begin{prop}\label{prop:SummaryCC}
For $\epsilon\in (0,1/7)$, our $\mathcal{O}\left(\frac{n^{2.5}|E|^{1.25}}{\epsilon^{2.5}}\log(n/\epsilon)\log(|E|)\right)$-time method outlined in Section~\ref{sec:CC} uses $\mathcal{O}(n+|E|)$ memory and generates a clustering of graph $G=(V,E)$ whose expected value satisfies $\mathbb{E}[\mathcal{C}] \geq 0.766(1-7\epsilon)\textup{opt}^G_{CC}$, where $\textup{opt}^G_{CC}$ is the optimal clustering value.
\end{prop}
The constructive proof outline of Propositions~\ref{prop:MkCSummary} and~\ref{prop:SummaryCC} is given in Sections~\ref{sec:MkC} and~\ref{sec:CC} respectively.

\textbf{Memory reduction using graph sparsification.} Propositions~\ref{prop:MkCSummary} and~\ref{prop:SummaryCC} state that the memory used by our approach is $\mathcal{O}(n+|E|)$. However, for dense graphs, the memory used by our method becomes $\Theta(n^2)$. In this setting, to reduce the memory used, we first need to change the way we access the problem instance. We assume that the input (weighted) graph $G$ arrives edge-by-edge, eliminating the need to store the entire dense graph. We then replace it with a $\tau$-spectrally close graph $\Tilde{G}$ (see Definition~\ref{def:specProp}) with $\mathcal{O}(n\log n/\tau^2)$ edges. Next, we compute an approximate solution to the new problem defined on the sparse graph using $\mathcal{O}(n\log n/\tau^2)$ memory.  For \textsc{Max-k-Cut} and \textsc{Max-Agree}, we show that this method generates a solution with provable approximation guarantees.

\subsection{Literature review}

We first review key low memory algorithms for linearly constrained SDPs.

\citet{burer} proposed a nonlinear programming approach which replaces the  PSD decision variable with its low-rank factorization in SDPs with $d$ linear constraints. If the selected value of rank $r$ satisfies $r(r+1)\geq 2d$ and the constraint set is a smooth manifold, then any second-order critical point of the nonconvex problem is a global optimum~\cite{boumal}.
Another approach, that requires $\Theta(d+nr)$ working memory, is to first determine (approximately) the subspace in which the (low) rank-$r$ solution to an SDP lies and then solve the problem over the (low) $r$-dimensional subspace~\cite{ding2019optimal}.

Alternatively, randomized sketching to a low dimensional subspace is often used as a low-memory alternative to storing a matrix decision variable~\cite{Woodruff,tropp2019streaming}. Recently, such sketched variables have been used to generate a low-rank approximation of a near-optimal solution to SDPs~\cite{yurtsever2019scalable}. The working memory required to compute a near-optimal solution and generate its rank-$r$ approximation using the algorithmic framework proposed by~\citet{yurtsever2019scalable} is $\mathcal{O}(d+rn/\zeta)$ for some sketching parameter $\zeta \in (0,1)$.
Gaussian sampling-based Frank-Wolfe~\cite{shinde2021memory} uses $\mathcal{O}(n+d)$ memory to generate zero-mean Gaussian samples whose covariance represents a near-optimal solution to the SDPs with $d$ linear constraints. This eliminates the dependency on the rank of the near-optimal solution or the accuracy to which its low rank approximation is computed. 

However, the two problems considered in this paper have SDP relaxations with $n^2$ constraints, for which applying the existing low-memory techniques provide no benefit since the memory requirement of these techniques depends on the number of constraints in the problem. These problems have been studied extensively in literature as we see below.

\paragraph{\textsc{Max-k-Cut}.} \textsc{Max-k-Cut} and its dual \textsc{Min-k-Partition} have applications in frequency allocation~\cite{eisenblatter2002frequency} and generating lower bound on co-channel interference in cellular networks~\cite{borndorfer1998frequency}. These problems have been studied extensively in the literature~\cite{kann1995hardness,newman2018complex,subramanian2008minimum}. The SDP-based rounding scheme given in~\cite{frieze1997improved} has also been adapted for similar problems of capacitated \textsc{Max-k-Cut}~\cite{gaur2008capacitated} and approximate graph coloring~\cite{karger1998approximate}. In each case, however, the SDP relaxation has $\Omega(n^2)$ constraints. Alternative heuristic methods have been proposed in~\cite{mladenovic1997variable,ma2017multiple,feo1995greedy,ghaddar2011branch}, however, these methods generate a feasible cut which only provides a lower bound on the optimal cut value.

\paragraph{Correlation clustering.} \citet{swamy2004correlation,charikar2005clustering} each provide 0.766-approximation schemes for \textsc{Max-Agree} which involve solving~\eqref{prob:CorrCluster}. For large-scale applications, data streaming techniques have been studied quite extensively for various clustering problems, such as $k$-means and $k$-median~\cite{ailon2009streaming,mccutchen2008streaming}. \citet{ahn2015correlation} propose a single-pass, $\mathcal{\Tilde{O}}(|E|+n\epsilon^{-10})$-time $0.766(1-\epsilon)$-approximation algorithm for \textsc{Max-Agree} that uses $\mathcal{\Tilde{O}}(n/\epsilon^2)$ memory. In contrast, to achieve the same approximation guarantee, our approach uses $\mathcal{O}\left(n+\min\left\{|E|,\frac{n\log n}{\epsilon^2}\right\}\right)$ memory which is equal to $\mathcal{O}(n+|E|)$ for sparse graphs,
%for sparse graphs (when $|E| \leq n\log n/\epsilon^2$)
and is independent of $\epsilon$. Furthermore, the computational complexity of our approach has a better dependence on $\epsilon$ given by 
%$\mathcal{O}\left(n^2\min\left\{|E|,\frac{n\log n}{\epsilon^2}\right\}\log (|E|+2n)\frac{1}{\epsilon^2}\right)$ which is at most $\mathcal{O}\left( \frac{n^3\log n\log (|E|+2n)}{\epsilon^4}\right)$.
$\mathcal{O}\left(\frac{n^{2.5}}{\epsilon^{2.5}}\min \{|E|,\frac{n\log n}{\epsilon^2}\}^{1.25}\log(n/\epsilon)\log(|E|)\right)$, and is at most $\mathcal{O}\left(\frac{n^{3.75}}{\epsilon^{5}}(\log n)^{1.25}\log(n/\epsilon)\log(|E|)\right)$.
Moreover, our approach is algorithmically simple to implement.

\subsection{Outline}
In Section~\ref{sec:Prelimnaries}, we review the Gaussian sampling-based Frank-Wolfe method~\cite{shinde2021memory} to compute a near-feasible, near-optimal solution to SDPs with linear equality and inequality constraints.
In Sections~\ref{sec:MkC} and~\ref{sec:CC} respectively, we adapt the Gaussian sampling-based approach to give an approximation algorithm for \textsc{Max-k-Cut} and \textsc{Max-Agree} respectively, that use only $\mathcal{O}(n+|E|)$ memory, proving Propositions~\ref{prop:MkCSummary} and~\ref{prop:SummaryCC} respectively. In Section~\ref{sec:sparsification}, we show how to combine our methods with streaming spectral sparsification to reduce the memory required to $\mathcal{O}(n\log n /\epsilon^2)$ for dense graphs presented edge-by-edge in a stream. We provide some preliminary computational results for \textsc{Max-Agree} in Section~\ref{sec:ComputationalResults}, and conclude our work and discuss possible future directions in Section~\ref{sec:Discussion}.
All proofs are deferred to the appendix.

\paragraph{Notations.}
The matrix inner product is denoted by $\left\langle A, B \right\rangle = \textup{Tr}\left(A^TB\right)$. The vector of diagonal entries of a matrix $X$ is $\textup{diag}(X)$, and $\textup{diag}^*(x)$ is a diagonal matrix with the vector $x$ on the diagonal. The notations $\mathcal{O}, \Omega, \Theta$ have the usual complexity interpretation and $\Tilde{\mathcal{O}}$ suppresses the dependence on $\log n$. An undirected edge $(i,j)$ in the set $E$ is denoted using $(i,j)\in E$ and $ij\in E$ interchangably.

\section{Gaussian Sampling-based Frank-Wolfe}\label{sec:Prelimnaries}
Consider a smooth, concave function $g$ and define the trace constrained SDP
\begin{equation}\label{prob:PrimalSDP}
\tag{BoundedSDP}
\max_{X\in \mathcal{S}}\quad g(\mathcal{B}(X)),
\end{equation}
where $\mathcal{S} = \{\textup{Tr}(X) \leq \alpha, X\succeq 0\}$ and $\mathcal{B}(\cdot): \mathbb{S}^{n} \rightarrow \mathbb{R}^d$ is a linear mapping that projects the variable from $\binom{n+1}{2}$-dimensional space to a $d$-dimensional space. One algorithmic approach to solving~\eqref{prob:PrimalSDP} is to use the Frank-Wolfe algorithm~\cite{FW} which, in this case, computes an $\epsilon$-optimal solution by taking steps of the form $X_{t+1} = (1-\gamma_t)X_t + \gamma_t \alpha h_th_t^T$, where $\gamma_t \in [0,1]$ and unit vectors $h_t$'s arise from approximately solving a symmetric eigenvalue problem that depends only on $\mathcal{B}(X_t)$ and $g(\cdot)$. Standard convergence results show that an $\epsilon$-optimal solution is reached after $\mathcal{O}(C^u_g/\epsilon)$ iterations, where $C_g^u$ is an upper bound on the curvature constant of $g$~\cite{FW}.

\paragraph{Frank-Wolfe with Gaussian sampling.}
The Gaussian sampling technique of~\cite{shinde2021memory} replaces the matrix-valued iterates, $X_t$, with Gaussian random vectors $z_t\sim \mathcal{N}(0,X_t)$. The update, at the level of samples, is then $z_{t+1} = \sqrt{1-\gamma_t}z_t+\sqrt{\gamma_t\alpha}\,\zeta_t h_t$, where $\zeta_t \sim \mathcal{N}(0,1)$. Note that $z_{t+1}$ is also a zero-mean Gaussian random vector with covariance equal to $X_{t+1} = (1-\gamma_t)X_t + \gamma_t \alpha h_th_t^T$. Furthermore, to track the change in the objective function value, it is sufficient to track the value $v_t = \mathcal{B}(X_t)$, and compute $v_{t+1}$ such that $v_{t+1} = (1-\gamma_t)v_t + \gamma_t \mathcal{B}(\alpha h_th_t^T)$.
Thus, computing the updates to the decision variable and tracking the objective function value only requires the knowledge of $z_t \sim \mathcal{N}(0,X_t)$ and $\mathcal{B}(X_t)$, which can be updated without explicitly storing $X_t$, thereby reducing the memory used.

Algorithm~\ref{Algo1}~\cite{shinde2021memory} describes, in detail, Frank-Wolfe algorithm with Gaussian sampling when applied to~\eqref{prob:PrimalSDP}. It uses at most $\mathcal{O}(n+d)$ memory at each iteration, and after at most $\mathcal{O}(C_g^u/\epsilon)$ iterations, returns a sample $\widehat{z}_{\epsilon}\sim\mathcal{N}(0,\widehat{X}_{\epsilon})$, where $\widehat{X}_{\epsilon}$ is an $\epsilon$-optimal solution to~\eqref{prob:PrimalSDP}.

\begin{algorithm2e}[ht] 
\caption{(\texttt{FWGaussian}) Frank-Wolfe Algorithm with Gaussian Sampling~\cite{shinde2021memory}} \label{Algo1}

\SetAlgoLined
\SetKwInOut{Input}{Input}\SetKwInOut{Output}{Output}
\DontPrintSemicolon

\Input{Input data for~\eqref{prob:PrimalSDP}, stopping criteria $\epsilon$, accuracy parameter $\eta$, upper bound $C_g^u$ on the curvature constant, probability $p$ for the subproblem \texttt{LMO}}

\Output{$z \sim \mathcal{N}(0,\widehat{X}_{\epsilon})$ and $v = \mathcal{B}(\widehat{X}_{\epsilon})$, where $\widehat{X}_{\epsilon}$ is an $\epsilon$-optimal solution of~\eqref{prob:PrimalSDP}}
\BlankLine

 \SetKwFunction{CDD}{LMO}
 \SetKwFunction{UV}{UpdateVariable}
 \SetKwProg{Fn2}{Function}{:}{\KwRet{$z_t$}}
 \SetKwFunction{FWGS}{FWGaussian}
 \SetKwProg{Fn}{Function}{:}{\KwRet{$(z_t,v_t)$}}
 \Fn{\FWGS}{
 Select initial point $X_0 \in \mathcal{S}$; set $v_0 \leftarrow \mathcal{B}(X_0)$ and 
 sample $z_{0} \sim \mathcal{N}(0,X_0)$\;
 $t \leftarrow 0$, $\gamma \leftarrow 2/(t+2)$\;
 $(h_t,q_t) \leftarrow \CDD (\mathcal{B}^*(\nabla g(v_t)), \frac{1}{2}\eta\gamma C_g^u)$\;

 \While{$\left\langle q_t - v_t, \nabla g(v_t)\right\rangle > \epsilon$}{
 $(z_{t+1}, v_{t+1}) \leftarrow \UV (z_{t},v_{t},h_t,q_t,\gamma)$\;
 
 $t \leftarrow t+1$, $\gamma \leftarrow 2/(t+2)$\;
 $(h_t,q_t) \leftarrow \CDD (\mathcal{B}^*(\nabla g(v_t)), \frac{1}{2}\eta\gamma C_g^u,p)$\;

 }
 }
 
 \SetKwProg{Fn}{Function}{:}{\KwRet{($h,q$)}}
 \Fn{\CDD{$J$, $\delta$}}{
 Find a unit vector $h$ such that with probability at least $1-p$, $\alpha \lambda = \alpha\langle hh^T, J\rangle \geq \max_{d\in \mathcal{S}} \alpha\langle d, J\rangle - \delta$\;
 \lIf{$\lambda \geq 0$}{
 $q \leftarrow \mathcal{B}(\alpha hh^T)$
 } \lElse{
 $q \leftarrow 0$, $h \leftarrow 0$
 }
 }
 
 \SetKwProg{Fn}{Function}{:}{\KwRet{($z,v$)}}
 \Fn{\UV{$z,v,h,q,\gamma$}}{
 $z \leftarrow (\sqrt{1-\gamma})z + \sqrt{\gamma\alpha}\, h\zeta\;\;\textup{where $\zeta\sim \mathcal{N}(0,1)$}$\;
 $v \leftarrow (1-\gamma)v + \gamma q $\;
 
 }

\end{algorithm2e}

\subsection{SDP with linear equality and inequality constraints}\label{sec:SDPLinEq}
Consider an SDP with linear objective function and a bounded feasible region,
\begin{equation}\label{prob:SDPIn}\tag{SDP}
\max_{X\succeq 0}\quad \langle C, X\rangle\quad \textup{subject to}\quad \begin{cases}
&\mathcal{A}^{(1)}(X) = b^{(1)}\\
&\mathcal{A}^{(2)}(X) \geq b^{(2)},
\end{cases}
\end{equation}
where $\mathcal{A}^{(1)}(\cdot):\mathbb{S}^n_+\rightarrow \mathbb{R}^{d_1}$ and $\mathcal{A}^{(2)}(\cdot):\mathbb{S}^n_+\rightarrow \mathbb{R}^{d_2}$ are linear maps. To use Algorithm~\ref{Algo1}, the linear constraints are penalized using a smooth penalty function. Let $u_l = \langle A^{(1)}_l,X\rangle - b^{(1)}_l$ for $l = 1,\dotsc,d_1$ and $v_l = b^{(2)}_l - \langle A^{(2)}_l,X\rangle$ for $l = 1,\dotsc,d_2$. 
For $M>0$, the smooth function $\phi_M(\cdot):\mathbb{R}^{d_1+d_2}\rightarrow \mathbb{R}$,
\begin{gather}
\label{eqn:penaltyPhiDef1}
\phi_M(u,v) = \frac{1}{M}\log\left(\sum_{i=1}^{d_1} e^{M(u_i)} + \sum_{i=1}^{d_1} e^{M(-u_i)} + \sum_{j=1}^{d_2} e^{M(v_j)}\right),\quad \textup{satisfies}\\
\label{eqn:LogPenBound1}
\max \left\{\|u\|_{\infty},\max_i v_i\right\} \leq \phi_M (u,v) \leq \frac{\log(2d_1+d_2)}{M} + \max \left\{\|u\|_{\infty}, \max_i v_i\right\}.
\end{gather}

We add this penalty term to the objective of~\eqref{prob:SDPIn} and define
\begin{equation}\label{prob:SDPInPenalized}\tag{SDP-LSE}
\max_{X\succeq 0}\left\{ \langle C,X\rangle - \beta \phi_M(\mathcal{A}^{(1)}(X) - b^{(1)},b^{(2)} - \mathcal{A}^{(2)}(X))\ : \ \textup{Tr}(X) \leq \alpha \right\},
\end{equation}
where $\alpha, \beta$ and $M$ are appropriately chosen parameters.
Algorithm~\ref{Algo1} then generates a Gaussian sample with covariance $\widehat{X}_{\epsilon}$ which is an $\epsilon$-optimal solution to~\eqref{prob:SDPInPenalized}. It is also a near-optimal, near-feasible solution to~\eqref{prob:SDPIn}. This result is a slight modification of~\cite[Lemma~3.2]{shinde2021memory} which only provides bounds for SDPs with linear equality constraints.

\begin{lemma}\label{lemma:SubOptFeasSDP}
For $\epsilon > 0$, let $(X^{\star},\vartheta^{\star},\mu^{\star})$ be an optimal primal-dual solution to~\eqref{prob:SDPIn} and its dual, and let $\widehat{X}_{\epsilon}$ be an $\epsilon$-optimal solution to~\eqref{prob:SDPInPenalized}. If $\beta > \| \vartheta^{\star}\|_1+\|\mu^{\star}\|_1$ and $M>0$, then
\begin{gather}
\langle C, X^{\star}\rangle - \frac{\beta\log(2d_1+d_2)}{M} - \epsilon \leq \langle C, \widehat{X}_{\epsilon}\rangle \leq \langle C, X^{\star}\rangle + (\| \vartheta^{\star}\|_1+\|\mu^{\star}\|_1) \frac{\beta \frac{\log(2d_1+d_2)}{M} + \epsilon}{\beta -  \| \vartheta^{\star}\|_1-\|\mu^{\star}\|_1},\\
\max \left\{
\| \mathcal{A}^{(1)}(X)-b^{(1)} \|_{\infty}, \max_i \left(b^{(2)}_i -\mathcal{A}^{(2)}_i(X)\right)\right\}  \leq \frac{\beta \frac{\log(2d_1+d_2)}{M} + \epsilon}{\beta -  \| \vartheta^{\star}\|_1-\|\mu^{\star}\|_1}.
\end{gather}
\end{lemma}

\section{Application of Gaussian Sampling to~\eqref{prob:maxkcutP}}\label{sec:MkC}
In this section, we look at the application of Gaussian sampling to \textsc{Max-k-Cut}. Since Algorithm~\ref{Algo1} uses $\mathcal{O}(n^2)$ memory when solving~\eqref{prob:maxkcutP}, we define a new SDP relaxation of \textsc{Max-k-Cut} with the same approximation guarantee, but with $\mathcal{O}(|E|)$ constraints. We then apply Algorithm~\ref{Algo1} to this new relaxation, and show how to round the solution to achieve nearly the same approximation ratio as given in~\eqref{eqn:maxcutbound}. Let
\begin{equation}\label{eqn:defalphak}
\alpha_k = \min_{-1/(k-1)\leq \rho \leq 1} \frac{kp(\rho)}{(k-1)(1-\rho)},
\end{equation}
where $p(X_{ij})$ is the probability that vertices $i$ and $j$ are in different partitions. If $X$ is feasible for~\eqref{prob:maxkcutP} and \texttt{CUT} is the value of the $k$-cut generated by the FJ rounding scheme, then
\begin{equation}\label{eqn:MkCboundDerivation}
\begin{split}
\mathbb{E}[\texttt{CUT}] &= \sum_{ij\in E,i<j}w_{ij}p(X_{ij})\\
&\geq \sum_{ij\in E,i<j}\frac{k-1}{k} w_{ij}(1-X_{ij})\alpha_k = \alpha_k \langle C, X\rangle.
\end{split}
\end{equation}
\citet{frieze1997improved} derive a lower bound on $\alpha_k$, showing that the method gives a nontrivial approximation guarantee. Observe that~\eqref{eqn:MkCboundDerivation} depends only on the values $X_{ij}$ if $(i,j)\in E$.

\paragraph{A new SDP relaxation of~\textsc{Max-k-Cut}.}
We relax the constraints in~\eqref{prob:maxkcutP} to define
\begin{equation}\label{prob:maxkcutApprox}\tag{k-Cut-Rel}
\max_{X\succeq 0}\quad\langle C, X\rangle\quad\textup{subject to}\quad \begin{cases}&\textup{diag}(X) =\mathbbm{1}\\
&X_{ij} \geq -\frac{1}{k-1}\quad (i,j)\in E, i<j.\end{cases}
\end{equation}
Since~\eqref{prob:maxkcutApprox} is a relaxation of~\eqref{prob:maxkcutP}, its optimal objective function value provides an upper bound on $\langle C, X^{\star}\rangle$, where $X^{\star}$ is an optimal solution to~\eqref{prob:maxkcutP}, and hence, on the optimal $k$-cut value $\textup{opt}_k^G$. Note that the bound in~\eqref{eqn:MkCboundDerivation} holds true even if we replace $X^{\star}$ by an optimal solution to~\eqref{prob:maxkcutApprox} since it depends on the value of $X_{ij}$ only if $(i,j)\in E$. Furthermore, when the FJ rounding scheme is applied to the solution of~\eqref{prob:maxkcutApprox}, it satisfies the approximation guarantee on the expected value of the generated $k$-cut given in~\eqref{eqn:maxcutbound}, i.e., $\mathbb{E}[\texttt{CUT}] \geq  \alpha_k \textup{opt}_k^G$.

\paragraph{Using Algorithm~\ref{Algo1}.}
We now have an SDP relaxation of \textsc{Max-k-Cut} that has $n+|E|$ constraints. Penalizing the linear constraints in~\eqref{prob:maxkcutApprox} using the function $\phi_M(\cdot)$~\eqref{eqn:penaltyPhiDef1}, Algorithm~\ref{Algo1} can now be used to generate $k$ samples with covariance $\widehat{X}_{\epsilon}$ which is an $\epsilon$-optimal solution to
\begin{equation}\label{prob:maxkCutLSE}\tag{k-Cut-LSE}
\textstyle{\max\limits_{X\succeq 0}\left\{ \langle C,X\rangle - \beta \phi_M\left(\textup{diag}(X) - \mathbbm{1}, -\frac{1}{k-1} - e_i^TXe_j\right):(i,j)\in E,\textup{Tr}(X) \leq n \right\}}.
\end{equation}

\paragraph{Optimality and feasibility results for~\eqref{prob:maxkcutApprox}.}
We now show that an $\epsilon$-optimal solution to~\eqref{prob:maxkCutLSE} is also a near-optimal, near-feasible solution to~\eqref{prob:maxkcutApprox}.

\begin{lemma}\label{lemma:SubOptFeasMkC}
For $\epsilon\in (0,1/2)$, let $X^{\star}_{R}$ be an optimal solution to~\eqref{prob:maxkcutApprox} and let $\widehat{X}_{\epsilon}$ be an $\epsilon\textup{Tr}(C)$-optimal solution to~\eqref{prob:maxkCutLSE}. For $\beta = 6\textup{Tr}(C)$ and $M = 6\frac{\log(2n+|E|)}{\epsilon}$, we have
\begin{gather}\label{eqn:OptBoundMkC}
(1-2\epsilon)\langle C, X^{\star}_{R}\rangle \leq \langle C, \widehat{X}_{\epsilon}\rangle \leq (1+4\epsilon)\langle C, X^{\star}_{R}\rangle \quad\textrm{and}\\
\label{eqn:FeasBoundMkC}
\| \textup{diag}(\widehat{X}_{\epsilon}) - \mathbbm{1} \|_{\infty} \leq \epsilon,\quad [\widehat{X}_{\epsilon}]_{ij}\geq -\frac{1}{k-1} - \epsilon,\quad(i,j)\in E,i<j.
\end{gather}
\end{lemma}

\paragraph{Generating a feasible solution to \textsc{Max-k-Cut}.} Since $\widehat{X}_{\epsilon}$ might not necessarily be feasible to~\eqref{prob:maxkcutApprox}, we cannot apply the FJ rounding scheme to the samples $z_i\sim \mathcal{N}(0,\widehat{X}_{\epsilon})$. We, therefore, generate samples $z^f_i\sim\mathcal{N}(0,X^f)$ using the procedure given in Algorithm~\ref{algo:GenkSamples} such that $X^f$ is a feasible solution to~\eqref{prob:maxkcutApprox} and $\langle C, X^f\rangle$ is close to $\langle C, \widehat{X}_{\epsilon}\rangle$.

\begin{algorithm2e}[htbp] 
\SetAlgoLined
\SetKwInOut{Input}{Input}\SetKwInOut{Output}{Output}
\DontPrintSemicolon

\Input{Sample $z_i \sim \mathcal{N}(0,\widehat{X}_{\epsilon})$ for $i = 1,\dotsc,k$ and $\textup{diag}(\widehat{X}_{\epsilon})$}
\Output{$z^f_i\sim \mathcal{N}(0,X^f)$ for $i = 1,\dotsc,k$ with $X^f$ feasible to~\eqref{prob:maxkcutApprox}}
\BlankLine

\SetKwFunction{GSk}{GeneratekSamples}
\SetKwProg{Fn}{Function}{:}{\KwRet{$z^f_1,\dotsc,z^f_k$}}

 \Fn{\GSk}{
 \For{$i = 1,\dotsc,k$}{
 Set $\textup{err} = \max \{ 0, \max_{(i,j)\in E,i<j}\ -1/(k-1) - [\widehat{X}_{\epsilon}]_{ij}\}$\;
 Set $\overline{z}_i = z_i + \sqrt{\textup{err}}\,y\,\mathbbm{1}$, where $y \sim \mathcal{N}(0,1)$\;
 Generate $\zeta \sim \mathcal{N}\left(0, I - \textup{diag}^*\left(\frac{\textup{diag}(\widehat{X}_{\epsilon})+\textup{err}}{\max (\textup{diag}(\widehat{X}_{\epsilon}))+\textup{err}}\right)\right)$\;
 Set $z^f_i = \frac{\overline{z}_i}{\sqrt{\max (\textup{diag}(\widehat{X}_{\epsilon}))+\textup{err}}} + \zeta$
 }
 }
\caption{\label{algo:GenkSamples}Generate Gaussian samples with covariance feasible to~\eqref{prob:maxkcutApprox}}
\end{algorithm2e}
We can now apply the FJ rounding scheme to $z^f_1,\dotsc,z^f_k$ as given in Lemma~\ref{lemma:MaxkCutapproxSol}.

\begin{lemma}\label{lemma:MaxkCutapproxSol}
For $G=(V,E)$, let $\textup{opt}_k^G$ be the optimal $k$-cut value and let $X_R^{\star}$ be an optimal solution to~\eqref{prob:maxkcutApprox}. For $\epsilon \in \left(0,\frac{1}{4}\right)$, let $\widehat{X}_{\epsilon} \succeq 0$ satisfy~\eqref{eqn:OptBoundMkC} and~\eqref{eqn:FeasBoundMkC}. Let $z^f_1,\dotsc,z^f_k$ be random vectors generated by Algorithm~\ref{algo:GenkSamples} with input $z_i,\dotsc,z_k \sim \mathcal{N}(0, \widehat{X}_{\epsilon})$ and let \texttt{CUT} denote the value of a $k$-cut generated by applying the FJ rounding scheme to $z^f_1,\dotsc,z^f_k$. For $\alpha_k$ as defined by~\eqref{eqn:defalphak}, we have
\begin{equation}\label{eqn:MkCoptimalCut}
\alpha_k (1-4\epsilon)\textup{opt}_k^G \leq\alpha_k (1-4\epsilon)\langle C, X_R^{\star}\rangle\leq \mathbb{E}[\texttt{CUT}] \leq \textup{opt}_k^G.
\end{equation}
\end{lemma}

\paragraph{Computational complexity of Algorithm~\ref{Algo1} when applied to~\eqref{prob:maxkCutLSE}.}
Finally, in Lemma~\ref{lemma:ConvMkC}, we provide the computational complexity of the method proposed in this section, which concludes the proof of Proposition~\ref{prop:MkCSummary}.
\begin{lemma}\label{lemma:ConvMkC}
When the method proposed in this section (Section~\ref{sec:MkC}), with $p = \frac{\epsilon}{T(n,\epsilon)}$ and $T(n,\epsilon) = \frac{144\log(2n+|E|)n^2}{\epsilon^2}$, is used to generate an approximate $k$-cut to \textsc{Max-k-Cut}, the generated cut satisfies $\mathbb{E}[\texttt{CUT}] \geq \alpha_k (1-5\epsilon)\textup{opt}_k^G$ and runs in $\mathcal{O}\left(\frac{n^{2.5}|E|^{1.25}}{\epsilon^{2.5}}\log(n/\epsilon)\log(|E|)\right)$ time.
\end{lemma}

\section{Application of Gaussian Sampling to~\eqref{prob:CorrCluster}}\label{sec:CC}
We now look at the application of our Gaussian sampling-based method to \textsc{Max-Agree}. Algorithm~\ref{Algo1} uses $\mathcal{O}(n^2)$ memory to generate samples whose covariance is an $\epsilon$-optimal solution to~\eqref{prob:CorrCluster}. However, with the similar observation as in the case of \textsc{Max-k-Cut}, we note that for any $X$ feasible to~\eqref{prob:CorrCluster}, the proof of the inequality $\mathbb{E}[\mathcal{C}] \geq 0.766 \langle C, X\rangle$, given in~\cite[Theorem~3]{charikar2005clustering}, requires $X_{ij}\geq 0$ only if $(i,j) \in E$. We therefore, write a new relaxation of~\eqref{prob:CorrCluster},
\begin{equation}\label{prob:CorrCluster1}\tag{MA-Rel}
\max_{X\succeq 0}\ \langle C, X\rangle = \langle  W^+ + L_{G^-}, X\rangle\ \ \ \textup{subject to}\ \begin{cases}
&X_{ii} = 1\ \ \forall\ i\in \{1,\dotsc,n\}\\
&X_{ij} \geq 0\ \ (i,j)\in E, i<j,
\end{cases}
\end{equation}
with only $n+|E|$ constraints.
The bound $\mathbb{E}[\mathcal{C}] \geq 0.766\langle C,X^{\star}\rangle\geq 0.766 \textup{opt}_{CC}^G$ on the expected value of the clustering holds even if the clustering is generated by applying the CGW rounding scheme to an optimal solution $X^{\star}$ of~\eqref{prob:CorrCluster1}. To use Algorithm~\ref{Algo1}, we penalize the constraints in~\eqref{prob:CorrCluster1} and define
\begin{equation}\label{prob:CC-LSE}\tag{MA-LSE}
\max_{X\succeq 0}\left\{ \langle C,X\rangle - \beta \phi_M(\textup{diag}(X) -\mathbbm{1}, -e_i^TXe_j):(i,j)\in E, \textup{Tr}(X) \leq n \right\}.
\end{equation}

\paragraph{Optimality and feasibility results for~\eqref{prob:CorrCluster1}.} Algorithm~\ref{Algo1} is now used to generate $z\sim \mathcal{N}(0,\widehat{X}_{\epsilon})$, where $\widehat{X}_{\epsilon}$ is an $\epsilon$-optimal solution to~\eqref{prob:CC-LSE}. We show in Lemma~\ref{lemma:SubFeasCC} that $\widehat{X}_{\epsilon}$ is also a near-optimal, near-feasible solution to~\eqref{prob:CorrCluster1}.

\begin{lemma}\label{lemma:SubFeasCC}
For $\Delta = \textup{Tr}(L_{G^-})+\sum_{ij\in E^+}w^+_{ij}$, $\epsilon\in \left(0,\frac{1}{4}\right)$, let $X^{\star}_G$ be an optimal solution to~\eqref{prob:CorrCluster1} and $\widehat{X}_{\epsilon}$ be an $\epsilon\Delta$-optimal solution to \eqref{prob:CC-LSE}. Setting $\beta = 4\Delta$ and $M = 4\frac{\log(2n+|E|)}{\epsilon}$, we have
\begin{gather}
(1-4\epsilon)\langle C, X^{\star}_G\rangle \leq \langle C, \widehat{X}_{\epsilon}\rangle \leq (1+4\epsilon)\langle C, X^{\star}_G\rangle\quad\textup{and}\label{eqn:OptBoundCC}\\
\| \textup{diag}(\widehat{X}_{\epsilon}) - \mathbbm{1} \|_{\infty} \leq \epsilon,\quad [\widehat{X}_{\epsilon}]_{ij} \geq -\epsilon,\quad(i,j)\in E,i<j\label{eqn:FeasBoundCC}.
\end{gather}
\end{lemma}

\paragraph{Generating an approximate clustering.}
The CGW rounding scheme can only be applied if we have a feasible solution to~\eqref{prob:CorrCluster1}.
We, therefore, use a modified version of Algorithm~\ref{algo:GenkSamples}, with Step 3 replaced by $\textup{err} = \max \{ 0, \max_{(i,j)\in E,i<j}\ - [\widehat{X}_{\epsilon}]_{ij}\}$ and input $z_1,z_2,z_3 \sim \mathcal{N}(0,\widehat{X}_{\epsilon})$, to generate zero-mean Gaussian samples whose covariance is a feasible solution to~\eqref{prob:CorrCluster1}. Finally, we apply the CGW rounding scheme to the output of the modified of Algorithm~\ref{algo:GenkSamples}.

\begin{lemma}\label{lemma:CC-ApproxSol}
Let $X^{\star}_G$ be an optimal solution to~\eqref{prob:CorrCluster1}. For $\epsilon \in \left(0,1/6\right)$, let $\widehat{X}_{\epsilon} \succeq 0$ satisfy~\eqref{eqn:OptBoundCC} and~\eqref{eqn:FeasBoundCC}, and let $z^f_1,z^f_2,z^f_3$ be random vectors generated by Algorithm~\ref{algo:GenkSamples} with input $z_1,z_2,z_3 \sim \mathcal{N}(0, \widehat{X}_{\epsilon})$. Let $\textup{opt}_{CC}^G$ denote the optimal clustering value for the graph $G=(V,E)$ and let $\mathcal{C}$ denote the value of the clustering generated from the random vectors $z^f_1,z^f_2,z^f_3$ using the CGW rounding scheme. Then
\begin{equation}\label{eqn:CC-ApproxCut}
\mathbb{E}[\mathcal{C}] \geq 0.766(1-6\epsilon) \langle C, X^{\star}_G\rangle \geq 0.766(1-6\epsilon)\textup{opt}_{CC}^G.
\end{equation}
\end{lemma}

\paragraph{Computational complexity of Algorithm~\ref{Algo1} when applied to~\eqref{prob:CC-LSE}.}
\begin{lemma}\label{lemma:ConvCC}
When the method proposed in this section (Section~\ref{sec:CC}), with $p = \frac{\epsilon}{T(n,\epsilon)}$ and $T(n,\epsilon) = \frac{64\log(2n+|E|)n^2}{\epsilon^2}$, is used to generate an approximate clustering, the value of the clustering satisfies $\mathbb{E}[\mathcal{C}] \geq 0.766(1-7\epsilon)\textup{opt}_{CC}^G$ and runs in $\mathcal{O}\left(\frac{n^{2.5}|E|^{1.25}}{\epsilon^{2.5}}\log(n/\epsilon)\log(|E|)\right)$ time.
\end{lemma}
This completes the proof of Proposition~\ref{prop:SummaryCC}.

\section{Sparsifying the Laplacian Cost Matrix}\label{sec:sparsification}
As seen in Sections~\ref{sec:MkC} and~\ref{sec:CC}, the memory requirement for generating and representing an $\epsilon$-optimal solution to~\eqref{prob:maxkCutLSE} and~\eqref{prob:CC-LSE} is bounded by  $\mathcal{O}(n+|E|)$. However, if the input graph $G$ is dense, the cost matrix will be dense and the number of inequality constraints will still be high.
In this section, we consider the situation in which the dense weighted graph arrives in a stream, and we first build a sparse approximation with similar spectral properties. We refer to this additional step as \textit{sparsifying} the cost.
\begin{definition}[$\tau$-spectral closeness]\label{def:specProp}
Two graphs $G$ and $\Tilde{G}$ defined on the same set of vertices are said to be \emph{$\tau$-spectrally close} if, for any $x\in \mathbb{R}^n$ and $\tau\in (0,1)$,
\begin{equation}
\label{eqn:SpecProp}
(1-\tau)x^TL_Gx \leq x^TL_{\Tilde{G}}x \leq (1+\tau)x^TL_Gx.
\end{equation}
\end{definition}

Spectral graph sparsification has been studied quite extensively (see, e.g., \cite{spielman2011spectral,ahn2009graph,fung2019general}). \citet{kyng2017framework} propose a $\mathcal{O}(|E|\log^2n)$-time framework to replace a dense graph $G = (V,E)$ by a sparser graph $\Tilde{G} = (V,\Tilde{E})$ such that $|\Tilde{E}| \sim \mathcal{O}(n\log n /\tau^2)$ and $\Tilde{G}$ satisfies~\eqref{eqn:SpecProp} with probability $1 - \frac{1}{\textup{poly}(n)}$. Their algorithm assumes that the edges of the graph arrive one at a time, so that the total memory requirement is $\mathcal{O}(n\log n /\tau^2)$ rather than $\mathcal{O}(|E|)$.
Furthermore, a sparse cost matrix decreases the computation time of the subproblem in Algorithm~\ref{Algo1} since it involves matrix-vector multiplication with the gradient of the cost.

\paragraph{\textsc{Max-k-Cut} with sparsification.} Let $\Tilde{G}$ be a sparse graph with $\mathcal{O}(n\log n/\tau^2)$ edges that is $\tau$-spectrally close to the input graph $G$. By applying the method outlined in Section~\ref{sec:MkC}, we can generate a $k$-cut for the graph $\Tilde{G}$ (using $\mathcal{O}(n\log n/\tau^2)$ memory) whose expected value satisfies the bound~\eqref{eqn:MkCoptimalCut}. Note that, this generated cut is also a $k$-cut for the original graph $G$ with provable approximation guarantee as shown in Lemma~\ref{lemma:ApproxCutMkCSparse}.

\begin{lemma}\label{lemma:ApproxCutMkCSparse}
For $\epsilon, \tau \in (0,1/5)$, let $\widehat{X}_{\epsilon}$ be a near-feasible, near-optimal solution to~\eqref{prob:maxkcutApprox} defined on the graph $\Tilde{G}$ that satisfies~\eqref{eqn:OptBoundMkC} and~\eqref{eqn:FeasBoundMkC}. Let \texttt{CUT} denote the value of the $k$-cut generated by applying Algorithm~\ref{algo:GenkSamples} followed by the FJ rounding scheme to $\widehat{X}_{\epsilon}$. Then the generated cut satisfies
\begin{equation}\label{eqn:MkCSparseCut}
\alpha_k (1-4\epsilon-\tau)\textup{opt}_k^G \leq \mathbb{E}[\texttt{CUT}]\leq \textup{opt}_k^G,
\end{equation}
where $\textup{opt}_k^G$ is the value of the optimal $k$-cut for the original graph $G$.
\end{lemma}

\begin{comment}
The proof is simple. The cost $C$ is a scaled Laplacian of the graph $G$ and hence the cost $\Tilde{C}$ for the sparse graph $\Tilde{G}$ satisfies~\eqref{eqn:SpecPropMatrix} at $X^{\star}_G$. So, we have $(1-\tau)\langle C, X^{\star}_G\rangle \leq \langle \Tilde{C}, X^{\star}_G\rangle \leq \langle \Tilde{C}, X^{\star}_{\Tilde{G}}\rangle$, where the final inequality follows since $X^{\star}_{\Tilde{G}}$ is an optimal solution to~\eqref{prob:maxkcutApprox} defined on the graph $\Tilde{G}$. Combining with~\eqref{eqn:MkCoptimalCut}, we get the above result.
\end{comment}

\paragraph{\textsc{Max-Agree} with sparsification.}
The number of edges $|E^+|$ and $|E^-|$ in  graphs $G^+$ and $G^-$ respectively determine the working memory of Algorithm~\ref{Algo1}. For dense input graphs $G^+$ and $G^-$, we sparsify them to generate graphs $\Tilde{G}^+$ and $\Tilde{G}^-$ with at most $\mathcal{O}(n\log n/\tau^2)$ edges and define
\begin{equation}\label{prob:CC-Sparse}\tag{MA-Sparse}
\max_{X\in \mathcal{S}} \Tilde{f}(X) = \langle L_{\Tilde{G}^-}+\Tilde{W}^+, X\rangle - \beta \phi_M(\textup{diag}(X)-\mathbbm{1},-[e_i^TXe_j]_{(i,j)\in \Tilde{E}}),
\end{equation}
where $\mathcal{S} =\{X:\textup{Tr}(X) \leq n, X\succeq 0\}$, $ L_{\Tilde{G}^-}$ is the Laplacian of the graph $\Tilde{G}^-$, $\Tilde{W}^+$ is matrix with nonnegative entries denoting the weight of each edge $(i,j)\in \Tilde{E}^+$, and $\Tilde{E} = \Tilde{E}^+ \cup \Tilde{E}^-$.
Algorithm~\ref{Algo1} then generates an $\epsilon(\textup{Tr}(L_{\Tilde{G}^-})+\sum_{ij\in \Tilde{E}^+}\Tilde{w}^+_{ij})$-optimal solution, $\widehat{X}_{\epsilon}$, to~\eqref{prob:CC-Sparse} using $\mathcal{O}(n\log n/\tau^2)$ memory. We can now use the method given in Section~\ref{sec:CC} to generate a clustering of graph $\Tilde{G}$ whose expected value, $\mathbb{E}[\mathcal{C}]$, satisfies~\eqref{eqn:CC-ApproxCut}. The following lemma shows that $\mathcal{C}$ also represents a clustering for the original graph $G$ with provable guarantees.

\begin{lemma}\label{lemma:ApproxCutCCSparse}
For $\epsilon, \tau \in (0,1/9)$, let  $\widehat{X}_{\epsilon}$ be a near-feasible, near-optimal solution to~\eqref{prob:CC-Sparse} defined on the graph $\Tilde{G}$ that satisfies~\eqref{eqn:OptBoundCC} and~\eqref{eqn:FeasBoundCC}. Let $\mathcal{C}$ denote the value of the clustering generated by applying Algorithm~\ref{algo:GenkSamples} followed by the CGW rounding scheme to $\widehat{X}_{\epsilon}$. Then, $\mathbb{E}[\mathcal{C}]$ satisfies
\begin{equation}\label{eqn:CCSparseCut}
0.766 (1-6\epsilon-3\tau)(1-\tau^2)\textup{opt}_{CC}^G \leq \mathbb{E}[\mathcal{C}]\leq \textup{opt}_{CC}^G,
\end{equation}
where $\textup{opt}_{CC}^G$ is the value of the optimal clustering of the original graph $G$.
\end{lemma}

We summarize our results in the following lemma whose proof is given in Appendix~\ref{appendix:SparseSummaryProof}.

\begin{lemma}\label{lemma:SummaryPropSparse}
Assume that the edges of the input graph $G=(V,E)$ arrive one at a time in a stream.
The procedure given in this section uses at most $\mathcal{O}(n\log n /\tau^2)$ memory and in $\mathcal{O}\left(\frac{n^{2.5}|E|^{1.25}}{\epsilon^{2.5}}\log(n/\epsilon)\log(|E|)\right)$-time, generates approximate solutions to \textsc{Max-k-Cut} and \textsc{Max-Agree} that satisfy the bounds $\mathbb{E}[\texttt{CUT}] \geq \alpha_k (1-5\epsilon-\tau)\textup{opt}_k^G$ and
$\mathbb{E}[\mathcal{C}] \geq 0.766 (1-7\epsilon-3\tau)(1-\tau^2)\textup{opt}_{CC}^G$
respectively. 
\end{lemma}

\section{Computational Results} \label{sec:ComputationalResults}
We now discuss the results of preliminary computations to cluster the vertices of a graph $G$ using the approach outlined in Section~\ref{sec:CC}. The aim of numerical experiments was to verify that the bounds given in Lemma~\ref{lemma:CC-ApproxSol} were satisfied when we used the procedure outlined in Section~\ref{sec:CC} to generate a clustering for each input graph. We used the graphs from \textsc{GSet} dataset~\cite{gset} which is a collection of randomly generated graphs. Note that the aim of correlation clustering is to  generate a clustering of vertices for graphs where each edge has a label indicating `similarity' or `dissimilarity' of the vertices connected to that edge. We, therefore, first converted the undirected, unweighted graphs from the \textsc{GSet} dataset~\cite{gset} into the instances of graphs with labelled edges using an adaptation of the approach used in~\cite{wang2013scalable,veldt2019metric}. This modified approach generated a label and weight for each edge $(i,j)\in E$ indicating the amount of `similarity' or `dissimilarity' between vertices $i$ and $j$.

\paragraph{Generating input graphs for \textsc{Max-Agree}.}
In the process of label generation, we first computed the Jaccard coefficient $J_{ij} = |N(i)\cap N(j)|/|N(i)\cup N(j)|$, where $N(i)$ is the set of neighbours of $i$ for each edge $(i,j)\in E$. Next we computed the quantity $S_{ij} = \log((1-J_{ij}+\delta)/(1+J_{ij}-\delta))$ with $\delta = 0.05$ for each edge $(i,j)\in E$, which is a measure of the amount of `similarity' or `dissimilarity'. Finally, the edge $(i,j)$ was labelled as `dissimilar' if $S_{ij}<0$ with $w^{-}_{ij} = -S_{ij}$ and labelled as `similar' with $w^+_{ij} = S_{ij}$ otherwise.

\paragraph{Experimental Setup.} We set the input parameters to $\epsilon = 0.05$, $\Delta = \textup{Tr}(L_{G^-})+\sum_{ij\in E^+}w^+_{ij}$, $\beta = 4\Delta$, $M = 4\frac{\log(2n)+|E|}{\epsilon}$. Using Algorithm~\ref{Algo1}, \eqref{prob:CC-LSE} was solved to $\epsilon\Delta$-optimality and, we computed feasible samples using Algorithm~\ref{algo:GenkSamples}. Finally, we generated two Gaussian samples and created at most four clusters by applying the 0.75-rounding scheme proposed by Swamy~\cite[Theorem~2.1]{swamy2004correlation}, for simplicity. The computations were performed using MATLAB R2021a on a machine with 8GB RAM. We noted the peak memory used by the algorithm using the \texttt{profiler} command in MATLAB.

The computational result for some randomly selected instances from the dataset is given in Table~\ref{table:ResultNoOpt}. We have provided the result for the rest of the graphs from \textsc{GSet} in Appendix~\ref{appendix:AdditionalResultsCC}. First, we observed that for each input graph, the number of iterations of \texttt{LMO} for $\epsilon\Delta$-convergence satisfied the bound given in Proposition~\ref{prop:MkCSummary} and the infeasibility of the covariance $\widehat{X}_{\epsilon}$ of the generated samples was less than $\epsilon$ satisfying~\eqref{eqn:FeasBoundCC}.
We generated 10 pairs of i.i.d. zero-mean Gaussian samples with covariance $\widehat{X}_{\epsilon}$, and each in turn was used to generate a clustering for the input graph using the 0.75-rounding scheme proposed by~\citet{swamy2004correlation}. Amongst the 10 clusterings generated for each graph, we picked the clustering with largest value denoted by $\mathcal{C}_{\textup{best}}$. Note that, $\mathcal{C}_{\textup{best}}\geq \mathbb{E}[\mathcal{C}]\geq 0.75(1-6\epsilon)\langle C, X^{\star}_G\rangle \geq 0.75\frac{1-6\epsilon}{1+4\epsilon}\langle C,\widehat{X}_{\epsilon}\rangle$, where the last inequality follows from combining~\eqref{eqn:CC-ApproxCut} with~\eqref{eqn:OptBoundCC}. Since we were able to generate the values, $\mathcal{C}_{\textup{best}}$ and $\langle C,\widehat{X}_{\epsilon}\rangle$, we noted that the weaker bound $\mathcal{C}_{\textup{best}}/\langle C,\widehat{X}_{\epsilon}\rangle = \textup{AR} \geq 0.75(1-6\epsilon)/(1+4\epsilon)$ was satisfied by every input graph when $\epsilon = 0.05$.

Table~\ref{table:ResultNoOpt} also shows the memory used by our method. Consider the dataset G1, for which the memory used by our method was $1526.35 \textup{kB} \approx 9.8\times (|V|+|E^+|+|E^-|)\times 8$, where a factor of 8 denotes that MATLAB requires 8 bytes to store a real number. Similarly, we observed that our method used at most $c\times (|V|+|E^+|+|E^-|)\times 8$ memory to generate clusters for other instances from \textsc{GSet}, where $c \leq 33$ for every instance of the input graph, showing that the memory used was linear in the size of the input graph.

\begin{table}[htbp]
{\footnotesize
\begin{center}
\caption{Result of generating a clustering of graphs from \textsc{GSet} using the method outlined in Section~\ref{sec:CC}. We have, $\textup{infeas} = \max \{ \| \textup{diag}(X) - 1\|_{\infty}, \max \{0, -[\widehat{X}_{\epsilon}]_{ij}\}\}$, $\textup{AR} = \mathcal{C}_{\textup{best}}/\langle C, \widehat{X}_{\epsilon}\rangle$ and $0.75(1-6\epsilon)/(1+4\epsilon) = 0.4375$ for $\epsilon = 0.05$.}\label{table:ResultNoOpt}
\begin{tabular}{|c|c|c| c| c| c| c| c| c|c|}
\hline
\parbox[c]{1.0cm}{\raggedright Dataset} & \parbox[c]{0.6cm}{\raggedright $|V|$} & \parbox[c]{0.8cm}{\raggedright $|E^+|$} & \parbox[c]{0.8cm}{\raggedright $|E^-|$} & \parbox[c]{1.4cm}{\raggedright \# Iterations ($\times 10^3$)} & \parbox[c]{1.2cm}{\raggedright$\textup{infeas}$} & \parbox[c]{1.0cm}{\raggedright $\langle C, \widehat{X}_{\epsilon}\rangle$} & \parbox[c]{0.8cm}{ $\mathcal{C}_{\textup{best}}$} & \parbox[c]{0.8cm}{\raggedright $\textup{AR}$} & \parbox[c]{1.4cm}{Memory required (in kB)}\\
\hline
G1&800&2453&16627&669.46&$10^{-3}$&849.48&643&0.757&1526.35\\
G11&800&817&783&397.2&$6\times 10^{-4}$&3000.3&2080&0.693&448.26\\
G14&800&3861&797&330.02&$8\times 10^{-4}$&542.55&469.77&0.866&423.45\\
G22&2000&115&19849&725.66&$10^{-3}$&1792.9&1371.1&0.764&1655.09\\
G32&2000&2011&1989&571.42&$9\times 10^{-4}$&7370&4488&0.609&1124\\
G43&1000&248&9704&501.31&$10^{-3}$&803.8&616.05&0.766&654.46\\
G48&3000&0&6000&9806.22&0.004&599.64&461.38&0.769&736.09\\
G51&1000&4734&1147&1038.99&0.001&676.21&446.29&0.66&517.09\\
G55&5000&66&12432&2707.07&0.002&1244.2&901.74&0.724&1281.03\\
G57&5000&4981&5019&574.5&0.005&18195&10292&0.565&812.78\\
\hline
\end{tabular}
\end{center}
}
\end{table}

\section{Discussion}\label{sec:Discussion}
In this paper, we proposed a Gaussian sampling-based optimization algorithm to generate approximate solutions to \textsc{Max-k-Cut}, and the \textsc{Max-Agree} variant of correlation clustering using $\mathcal{O}\left(n+\min \left\{|E|,\frac{n\log n}{\tau^2} \right\}\right)$ memory. The approximation guarantees given in~\cite{frieze1997improved, charikar2005clustering, swamy2004correlation} for these problems are based on solving SDP relaxations of these problems that have $n^2$ constraints.
The key observation that led to the low-memory method proposed in this paper was that the approximation guarantees from literature are preserved for both problems even if we solve their weaker SDP relaxations with only $\mathcal{O}(n+|E|)$ constraints. We showed that for \textsc{Max-k-Cut}, and the \textsc{Max-Agree} variant of correlation clustering, our approach nearly preserves the quality of the solution as given in~\cite{frieze1997improved,charikar2005clustering}. We also implemented the method outlined in Section~\ref{sec:CC} to generate approximate clustering for random graphs with provable guarantees. The numerical experiments showed that while the method was simple to implement, it was slow in practice. However, there is scope for improving the convergence rate of our method so that it can potentially be applied to the large-scale instances of various real-life applications of clustering.

\paragraph{Extending the low-memory method to solve problems with triangle inequalities.} The known nontrivial approximation guarantees for sparsest cut problem involve solving an SDP relaxation that has $n^3$ triangle inequalities~\cite{arora2009expander}. It would be interesting to see whether it is possible to simplify these SDPs in such a way that they can be combined nicely with memory efficient algorithms, and still maintain good approximation guarantees.

{\fontsize{9}{9}\selectfont \bibliography{refLong,references}}

\appendix
\section{Proofs}

\subsection{Proof of Lemma~\ref{lemma:SubOptFeasSDP}}\label{appendix:OptFeasSDP}
\begin{proof}
Let $\vartheta \in \mathbb{R}^{d_1}$ and $\mu\in \mathbb{R}^{d_2}$ be the dual variables corresponding to the $d_1$ equality constraints and the $d_2$ inequality constraints respectively. The dual of~\eqref{prob:SDPIn} is
\begin{equation}\label{prob:SDP-Dual}
\tag{DSDP}
\min_{\vartheta,\mu}\quad \sum_{i=1}^{d_1} b^{(1)}_i\vartheta_i +  \sum_{j=1}^{d_2} b^{(2)}_j\mu_j\quad\textup{subject to}\quad \begin{cases}&\sum\limits_{i=1}^{d_1} \vartheta_i A^{(1)}_i + \sum\limits_{j=1}^{d_2} A^{(2)}_j \mu_j - C \succeq 0\\
&\mu \leq 0,\end{cases}
\end{equation}
where $A^{(2)}_j$'s for $j=1,\dotsc, d_2$ are assumed to be symmetric.

\paragraph{Lower bound on the objective.}
Let $X^{\star}$ be an optimal solution to~\eqref{prob:SDPIn} and let $X^\star_{FW}$ be an optimal solution to~\eqref{prob:SDPInPenalized}. For ease of notation, let
\begin{equation}\label{eqn:ProofL1Eqn1}
u = \mathcal{A}^{(1)}(X) - b^{(1)} \quad \textrm{and}\quad v =  b^{(2)} - \mathcal{A}^{(2)}(X),
\end{equation}
and define $(\widehat{u}_{\epsilon},\widehat{v}_{\epsilon})$, $(u_{FW},v_{FW})$ and $(u^{\star},v^{\star})$ by substituting $\widehat{X}_{\epsilon}$, $X_{FW}$ and $X^{\star}$ respectively in~\eqref{eqn:ProofL1Eqn1}. Since $\widehat{X}_{\epsilon}$ is an $\epsilon$-optimal solution to~\eqref{prob:SDPInPenalized}, and a feasible solution to~\eqref{prob:SDPInPenalized}, the following holds
\begin{equation}\label{eqn:SubOptimalSolPenaltyProb}
\langle C, \widehat{X}_{\epsilon}\rangle - \beta \phi_M(\widehat{u}_{\epsilon},\widehat{v}_{\epsilon}) \geq \langle C, X_{FW}\rangle - \beta \phi_M(u_{FW},v_{FW})  - \epsilon \geq \langle C, X^{\star}\rangle - \beta \phi_M(u^{\star},v^{\star})  - \epsilon.
\end{equation}
We know that $(u^{\star},v^{\star})$ is feasible to~\eqref{prob:SDPIn}, so that $\phi_M(u^{\star},v^{\star}) \leq \frac{\log(2d_1+d_2)}{M}$.
Now, rearranging the terms, and using the upper bound on $\phi_M(u^{\star},v^{\star})$ and the fact that $\phi_M(\widehat{u}_{\epsilon},\widehat{v}_{\epsilon})\geq 0$,
\begin{equation}\label{eqn:SubOptimalSolPenaltyProb2}
\langle C, \widehat{X}_{\epsilon}\rangle \geq \langle C, X^{\star}\rangle - \frac{\beta\log(2d_1+d_2)}{M} - \epsilon.
\end{equation}

\paragraph{Upper bound on the objective.} The Lagrangian of~\eqref{prob:SDPIn} is
$L(X,\vartheta,\mu) = \langle C,X\rangle - \sum_{i=1}^{d_1} u_i\vartheta_i + \sum_{j= 1}^{d_2} v_j\mu_j$.
For a primal-dual optimal pair, ($X^{\star},\vartheta^{\star},\mu^{\star}$), and $\widehat{X}_{\epsilon}\succeq 0$, we have that $L(\widehat{X}_{\epsilon},\vartheta^{\star},\mu^{\star})\leq L(X^{\star},\vartheta^{\star},\mu^{\star})$, i.e.,
\begin{equation*}
\begin{split}
\langle C, \widehat{X}_{\epsilon} \rangle - \sum_{i=1}^{d_1}\vartheta_{i}^{\star}[\widehat{u}_{\epsilon}]_{i} + \sum_{i= j}^{d_2} \mu_{j}^{\star}[\widehat{v}_{\epsilon}]_{j} &\leq \langle C, X^{\star} \rangle - \sum_{i=1}^{d_1} \vartheta^{\star}_{i}u^{\star}_{i} + \sum_{j= 1}^{d_2} \mu^{\star}_{j}v^{\star}_{j}\\
&\leq \langle C, X^{\star}\rangle.
\end{split}
\end{equation*}
Rearranging the terms, using the duality of the $\ell_1$ and $\ell_\infty$ norms, and the fact that $\mu^{\star} \leq 0$, gives
\begin{equation}\label{eqn:upperbound}
\begin{split}
\langle C, \widehat{X}_{\epsilon} \rangle &\leq \langle C, X^{\star} \rangle + \sum_{i=1}^{d_1} \vartheta^{\star}_{i}[\widehat{u}_{\epsilon}]_{i} - \sum_{j= 1}^{d_2} \mu^{\star}_{j}[\widehat{v}_{\epsilon}]_{j}\\
&\leq \langle C, X^{\star}\rangle + \left(\sum_{i=1}^{d_1} |\vartheta^{\star}_{i}|\right)\|\widehat{u}_{\epsilon}\|_{\infty} + \left(\sum_{j= 1}^{d_2} -\mu^{\star}_{j}\right)\max_j [\widehat{v}_{\epsilon}]_{j}\\
&\leq \langle C, X^{\star} \rangle + \|[\vartheta^{\star},\mu^{\star}]]\|_1 \max \left\{ \|\widehat{u}_{\epsilon}\|_{\infty}, \max_j [\widehat{v}_{\epsilon}]_{j} \right\}.
\end{split}
\end{equation}

\paragraph{Bound on infeasibility.}

Using~\eqref{eqn:upperbound}, we rewrite~\eqref{eqn:SubOptimalSolPenaltyProb} as,
\begin{equation}\label{eqn:ProofLemma11}
\begin{split}
\beta \phi_M(\widehat{u}_{\epsilon},\widehat{v}_{\epsilon}) &\leq \langle C, \widehat{X}_{\epsilon} \rangle - \langle C, X^{\star}\rangle +  \beta \phi_M(u^{\star},v^{\star}) + \epsilon\\
&\leq \|[\vartheta^{\star},\mu^{\star}]\|_1 \max \left\{ \|\widehat{u}_{\epsilon}\|_{\infty}, \max_j [\widehat{v}_{\epsilon}]_{j} \right\} +
\beta \frac{\log(2d_1+d_2)}{M} + \epsilon.
\end{split}
\end{equation}
Combining the lower bound on $\phi_M(\cdot)$ given in~\eqref{eqn:LogPenBound1} with~\eqref{eqn:ProofLemma11} and since $\beta > \| [\vartheta^{\star},\mu^{\star}] \|_1$ by assumption, we have
\begin{equation}
\max \left\{ \|\widehat{u}_{\epsilon}\|_{\infty}, \max_j [\widehat{v}_{\epsilon}]_{j} \right\} \leq \frac{\beta \frac{\log(2d_1+d_2)}{M} + \epsilon}{\beta - \| [\vartheta^{\star},\mu^{\star}]\|_1}. \label{eqn:BoundonInfeas}
\end{equation}

\paragraph{Completing the upper bound on the objective.} 
Substituting~\eqref{eqn:BoundonInfeas}  into~\eqref{eqn:upperbound} gives
\begin{equation}\label{eqn:upperboundSDP}
\langle C, \widehat{X}_{\epsilon} \rangle \leq \langle C, X^{\star}\rangle + \|[\vartheta^{\star},\mu^{\star}]\|_1 \frac{\beta \frac{\log(2d_1+d_2)}{M} + \epsilon}{\beta - \| [\vartheta^{\star},\mu^{\star}] \|_1}.
\end{equation}

\end{proof}

\subsection{Proof of Lemma~\ref{lemma:SubOptFeasMkC}}\label{appendix:OptFeas}
\begin{proof}
The proof consists of three parts.

\paragraph{Lower bound on the objective.}
Substituting the values of $\beta$ and $M$, and replacing $\epsilon$ by $\epsilon \textup{Tr}(C)$ in~\eqref{eqn:SubOptimalSolPenaltyProb2}, we have
\begin{equation}\label{eqn:SubOptimalSolPenaltyProb5MkC}
\langle C, \widehat{X}_{\epsilon}\rangle \geq \langle C, X^{\star}_R\rangle - 2\epsilon \textup{Tr}(C).
\end{equation}
Since the identity matrix $I$ is strictly feasible for~\eqref{prob:maxkcutApprox}, $\textup{Tr}(C) \leq \langle C, X^{\star}_{R}\rangle$. Combining this fact with~\eqref{eqn:SubOptimalSolPenaltyProb5MkC} gives,
\begin{equation*}
\langle C, \widehat{X}_{\epsilon}\rangle \geq (1-2\epsilon)\langle C, X^{\star}_{R}\rangle.
\end{equation*}

\paragraph{Bound on infeasibility.}
For~\eqref{prob:maxkcutApprox}, let $\nu= [\nu^{(1)},\nu^{(2)}]\in \mathbbm{R}^{n+|E|}$ be a dual variable such that $\nu^{(1)}_i$ for $i = 1,\dotsc,n$ are the variables corresponding to $n$ equality constraints and $\nu^{(2)}_{ij}$ for $(i,j)\in E, i<j$ are the dual variables corresponding to $|E|$ inequality constraints. Following~\eqref{prob:SDP-Dual}, the dual of~\eqref{prob:maxkcutApprox} is
\begin{equation}\label{prob:maxkcutApproxDual}
\tag{Dual-Relax}
\min_{\nu}\sum\limits_{i=1}^n \nu^{(1)}_{i} - \frac{1}{k-1} \sum_{\substack{ij \in E\\i<j}}\nu^{(2)}_{ij}\quad\textup{subject to}\quad \begin{cases}&\textup{diag}^*(\nu^{(1)}) + \sum\limits_{\substack{ij \in E\\i<j}} [e_ie_j^T+e_je_i^T]\frac{\nu^{(2)}_{ij}}{2} - C \succeq 0\\
&\nu^{(2)} \leq 0.\end{cases}
\end{equation}
Let $\nu^{\star}$ be an optimal dual solution.
In order to bound the infeasibility using~\eqref{eqn:BoundonInfeas}, we need a bound on $\|\nu^{\star}\|_1$ which is given by the following lemma.

\begin{lemma}\label{lemma:BoundonDualVar}
The value of $\|\nu^{\star}\|_1$ is upper bounded by $4\textup{Tr}(C)$.
\end{lemma}

\begin{proof}
The matrix $C$ is a scaled Laplacian and so, the only off-diagonal entries that are nonzero correspond to $(i,j)\in E$ and have value less than zero. For~\eqref{prob:maxkcutApproxDual}, a feasible solution is $\nu^{(1)} = \textup{diag}(C)$, $\nu^{(2)}_{ij} = 2C_{ij}$ for $(i,j)\in E, i<j$. The optimal objective function value of~\eqref{prob:maxkcutApproxDual} is then upper bounded by
\begin{align}
&\sum_{i=1}^n \nu^{(1)\star}_{i} - \frac{1}{k-1} \sum_{\substack{ij \in E\\i<j}}\nu^{(2)\star}_{ij} \leq \textup{Tr}(C) + \frac{1}{k-1}\textup{Tr}(C) = \frac{k}{k-1}\textup{Tr}(C)\\
\label{eqn:l1bound2MkC}
\Rightarrow\quad &\sum_{i=1}^n \nu^{(1)\star}_{i} \leq \frac{k}{k-1}\textup{Tr}(C) + \frac{1}{k-1} \sum_{\substack{ij \in E\\i<j}}\nu^{(2)\star}_{ij}  \leq \frac{k}{k-1}\textup{Tr}(C),
\end{align}
where the last inequality follows since $\nu^{(2)}\leq 0$.

We have $\left\langle 
\textup{diag}^*(\nu^{(1)\star}) + \sum\limits_{\substack{ij \in E\\i<j}} [e_ie_j^T+e_je_i^T]\frac{\nu^{(2)\star}_{ij}}{2}, \mathbbm{1}\mathbbm{1}^T\right\rangle - \langle C,\mathbbm{1}\mathbbm{1}^T\rangle \geq 0$ since both matrices are PSD. Using the fact that $\mathbbm{1}$ is in the null space of $C$, we get
\begin{equation}\label{eqn:l1boundMkC5}
- \sum_{\substack{ij \in E\\i<j}}\nu^{(2)\star}_{ij} \leq \sum_{i=1}^n \nu^{(1)\star}_{i}.
\end{equation}
Since $\nu^{(2)\star}\leq 0$, we can write
\begin{equation}\label{eqn:l1bound1MkC}
\|\nu^{\star}\|_1 =  \sum_{i=1}^{n}|\nu^{(1)\star}_{i}| - \sum_{\substack{ij \in E\\i<j}}\nu^{(2)\star}_i \leq 2\sum_{i=1}^{n}\nu^{(1)\star}_{i},
\end{equation}
which follows from~\eqref{eqn:l1boundMkC5} and the fact that for the dual to be feasible we have $\nu^{(1)}\geq 0$ since $C$ has nonnegative entries on the diagonal. Substituting~\eqref{eqn:l1bound2MkC} in~\eqref{eqn:l1bound1MkC},
\begin{equation}\label{eqn:l1bound3MkC}
\|\nu^{\star}\|_1 \leq \frac{2k}{k-1}\textup{Tr}(C)\leq 4\textup{Tr}(C),
\end{equation}
where the last inequality follows since $k/(k-1) \leq 2$ for $k\geq 2$.
\end{proof}

Since $\widehat{X}_{\epsilon}$ is an $\epsilon\textup{Tr}(C)$-optimal solution to~\eqref{prob:maxkCutLSE}, we replace $\epsilon$ be $\epsilon\textup{Tr}(C)$ in~\eqref{eqn:BoundonInfeas}. Finally, substituting~\eqref{eqn:l1bound3MkC} into~\eqref{eqn:BoundonInfeas}, and setting $\beta = 6\textup{Tr}(C)$ and $M = 6\frac{\log(2n+|E|)}{\epsilon}$,
\begin{equation}\label{eqn:BoundonInfeas1MkC}
\max \left\{ \|\textup{diag}(\widehat{X}_{\epsilon}) - \mathbbm{1}\|_{\infty}, \max_{ij\in E,i<j} -\frac{1}{k-1} - [\widehat{X}_{\epsilon}]_{ij} \right\} \leq \epsilon.
\end{equation}
This condition can also be stated as
\begin{equation*}
\|\textup{diag}(\widehat{X}_{\epsilon}) - \mathbbm{1}\|_{\infty} \leq \epsilon,\quad [\widehat{X}_{\epsilon}]_{ij} \geq -\frac{1}{k-1} - \epsilon\quad (i,j)\in E, i<j.
\end{equation*}

\paragraph{Upper bound on the objective.}
Substituting~\eqref{eqn:BoundonInfeas1MkC} and~\eqref{eqn:l1bound3MkC} and the values of parameters $\beta$ and $M$ into~\eqref{eqn:upperboundSDP} gives
\begin{equation*}
\langle C, \widehat{X}_{\epsilon} \rangle \leq \langle C, X^{\star}_{R}\rangle + 4\textup{Tr}(C)\epsilon \leq (1+4\epsilon)\langle C, X^{\star}_{R}\rangle,
\end{equation*}
where the last inequality follows since $\textup{Tr}(C) \leq \langle C, X^{\star}_{R}\rangle$.
\end{proof}

\subsection{Proof of Lemma~\ref{lemma:MaxkCutapproxSol}}
\begin{proof}
We first show that Algorithm~\ref{algo:GenkSamples} generates random Gaussian samples whose covariance is feasible to~\eqref{prob:maxkcutApprox}.
\begin{prop}\label{prop:AlgoGenSamplesProof}
Given $k$ Gaussian random vectors $z_1,\dotsc,z_k\sim \mathcal{N}(0,\widehat{X}_{\epsilon})$, such that their covariance $\widehat{X}_{\epsilon}$ satisfies the inequality~\eqref{eqn:FeasBoundMkC}, the Gaussian random vectors $z^f_1,\dotsc,z^f_k\sim \mathcal{N}(0,X^f)$ generated by Algorithm~\ref{algo:GenkSamples} have covariance $X^f$ that is a feasible solution to~\eqref{prob:maxkcutApprox}.
\end{prop}

\begin{proof}
Define $\overline{X} = \widehat{X}_{\epsilon} +\textup{err} \mathbbm{1}\mathbbm{1}^T$. Note that, $\overline{X}\succeq 0$ and it satisfies the following properties:
\begin{enumerate}
\item Since $\widehat{X}_{\epsilon}$ satisfies~\eqref{eqn:FeasBoundMkC}, we have $\textup{err} \leq \epsilon$. Combining this fact with the definition of $\overline{X}$, we have $\overline{X}_{jl} \geq -\frac{1}{k-1}$ for $(j,l)\in E,j<l$.
\item Furthermore, $\textup{diag}(\overline{X}) = \textup{diag}(\widehat{X}_{\epsilon})+\textup{err}$, which when combined with~\eqref{eqn:FeasBoundMkC}, gives $1\leq\textup{diag}(\overline{X})\leq 1+2\textup{err}$.
\item For $y\sim \mathcal{N}(0,1)$, if $\overline{z}_i = z_i+\sqrt{\textup{err}}y\mathbbm{1}$, i.e., it is a sum of two Gaussian random vectors, then $\overline{z}_i\sim \mathcal{N}(0,\overline{X})$.
\end{enumerate}

The steps 5 and 6 of Algorithm~\ref{algo:GenkSamples} generate a zero-mean random vector $z^f$ whose covariance is
\begin{equation}\label{eqn:FeasSolMkC}
X^f = \frac{\overline{X}}{\textup{max}(\textup{diag}(\overline{X}))} + \left(I - \textup{diag}^*\left(\frac{\textup{diag}(\overline{X})}{\max (\textup{diag}(\overline{X}))}\right)\right),
\end{equation}
i.e., $z^f \sim \mathcal{N}(0,X^f)$. Furthermore, $X^f$ is feasible to~\eqref{prob:maxkcutApprox} since $\textup{diag}(X^f) = \mathbbm{1}$, $X^f_{jl} \geq -\frac{1}{k-1}$ for $(j,l)\in E,j<l$, and it is a sum of two PSD matrices so that $X^f\succeq 0$.
\end{proof}

The objective function value of~\eqref{prob:maxkcutApprox} at $X^f$ (defined in~\eqref{eqn:FeasSolMkC}) is
\begin{align}
\langle C, X^f\rangle &= \left\langle C, \frac{\widehat{X}_{\epsilon}+\textup{err} \mathbbm{1}\mathbbm{1}^T}{\max (\textup{diag}(\widehat{X}_{\epsilon}))+\textup{err}} + \left(I - \textup{diag}^*\left(\frac{\textup{diag}(\widehat{X}_{\epsilon})+\textup{err}}{\max (\textup{diag}(\widehat{X}_{\epsilon}))+\textup{err}}\right)\right)\right\rangle\\
&\underset{(i)}{\geq} \ \frac{\langle C, \widehat{X}_{\epsilon}\rangle}{\max (\textup{diag}(\widehat{X}_{\epsilon}))+\textup{err}} \
\underset{(ii)}{\geq} \frac{1-2\epsilon}{1+2\epsilon}\langle C, X^{\star}_{R}\rangle \
\underset{(iii)}{\geq} \ (1-4\epsilon)\langle C, X^{\star}_{R}\rangle, \label{eqn:MkCCutboundproof1}
\end{align}
where (i) follows from the fact that both $C$ and $\frac{\textup{err} \mathbbm{1}\mathbbm{1}^T}{\max (\textup{diag}(\widehat{X}_{\epsilon}))+\textup{err}} + I -  \textup{diag}^*\left(\frac{\textup{diag}(\widehat{X}_{\epsilon})+\textup{err}}{\max (\textup{diag}(\widehat{X}_{\epsilon}))+\textup{err}}\right)$ are PSD and so, their inner product is nonnegative, (ii) follows from Lemma~\ref{lemma:SubOptFeasMkC} and the fact that $\textup{err}\leq \epsilon$, and~(iii) uses the fact that
$1-2\epsilon\geq (1+2\epsilon)(1-4\epsilon)$. Let $\mathbb{E}[\texttt{CUT}]$ denote the value of the cut generated from the samples $z^f_i$'s. Combining~\eqref{eqn:MkCCutboundproof1} with the inequality $\frac{\mathbb{E}[\texttt{CUT}]}{\langle C, X^f\rangle}\geq \alpha_k$ (see~\eqref{eqn:MkCboundDerivation}), we have
\begin{equation}\label{eqn:MkCCutboundproof2}
\mathbb{E}[\texttt{CUT}] \geq \alpha_k \langle C, X^f\rangle\geq \alpha_k (1-4\epsilon)\langle C, X^{\star}_{R}\rangle\geq \alpha_k (1-4\epsilon)\textup{opt}_k^G.
\end{equation}
\end{proof}

\subsection{Proof of Lemma~\ref{lemma:ConvMkC}}
\begin{proof}
We use Algorithm~\ref{Algo1} with $p = \frac{\epsilon}{T(n,\epsilon)}$ and $T(n,\epsilon) = \frac{144\log(2n+|E|)n^2}{\epsilon^2}$ to generate an $\epsilon\textup{Tr}(C)$-optimal solution to~\eqref{prob:maxkCutLSE}.
We first bound the outer iteration complexity, i.e., the number of iterations of Algorithm~\ref{Algo1} until convergence. This value also denotes the number of times the subproblem \texttt{LMO} is solved.

\paragraph{Upper bound on outer iteration complexity.}
Let the objective function of~\eqref{prob:maxkCutLSE} be $g(X) = \langle C,X\rangle - \beta \phi_M\left(\textup{diag}(X) - \mathbbm{1}, \left[-\frac{1}{k-1} - e_i^TXe_j\right]_{(i,j)\in E}\right)$.

\begin{theorem} \label{thm:FWapproxConv}
Let $g(X)$ be a concave and differentiable function and $X^{\star}$ an optimal solution of~\eqref{prob:maxkCutLSE}. Let $C_g^u$ be an upper bound on the curvature constant of $g$, and let $\eta \geq 0$ be the accuracy parameter for \texttt{LMO}, then, $X_t$ satisfies
\begin{equation}
-g(X_t) + g(X^{\star}) \leq \frac{2C_g^u(1+\eta)}{t+2},
\end{equation}
with probability at least $(1-p)^t \geq 1-tp$.
\end{theorem}

The result follows from \cite[Theorem~1]{FW} when \texttt{LMO} generates a solution with approximation error at most $\frac{1}{2}\eta\gamma C_g^u$ with probability $1-p$.
Now, $\eta \in (0,1)$ is an appropriately chosen constant, and from~\cite[Lemma~3.1]{shinde2021memory}, an upper bound $C_f^u$ on the curvature constant of $g(X)$ is $\beta M n^2$. Thus, after at most
\begin{equation}\label{eqn:BoundonT}
T = \frac{2C_g^u(1+\eta)}{\epsilon\textup{Tr}(C)} - 2 = \frac{2\beta M n^2(1+\eta)}{\epsilon\textup{Tr}(C)} - 2
\end{equation}
iterations, Algorithm~\ref{Algo1} generates an $\epsilon\textup{Tr}(C)$-optimal solution to~\eqref{prob:maxkCutLSE}.

\paragraph{Bound on the approximate $k$-cut value.} From Theorem~\ref{thm:FWapproxConv}, we see that after at most $T$ iterations, Algorithm~\ref{Algo1} generates a solution $\widehat{X}_{\epsilon}$ that satisfies the bounds in Lemma~\ref{lemma:SubOptFeasMkC} with probability with at least $1-\epsilon$ when $p = \frac{\epsilon}{T(n,\epsilon)}$. Consequently, the bound given in~\eqref{eqn:MkCCutboundproof1} also holds with probability at least $1-\epsilon$. And so, the expected value of $\langle C, X^f\rangle$ is $\mathbb{E}[\langle C, X^f\rangle] \geq (1-4\epsilon)\langle C, X^{\star}_{R}\rangle (1-\epsilon) \geq (1-5\epsilon)\langle C, X^{\star}_{R}\rangle$. Finally, from~\eqref{eqn:MkCCutboundproof2}, the expected value of the $k$-cut, denoted by $\mathbb{E}[\texttt{CUT}]$, is bounded as
\begin{equation}
\mathbb{E}[\texttt{CUT}] = \mathbb{E}_L[\mathbb{E}_G[\texttt{CUT}]] \geq \alpha_k \mathbb{E}_L[\langle C, X^f\rangle] \geq \alpha_k (1-5\epsilon)\langle C, X^{\star}_{R}\rangle\geq \alpha_k (1-5\epsilon)\textup{opt}_k^G,
\end{equation}
where $\mathbb{E}_L[\cdot]$ denotes the expectation over the randomness in the subproblem \texttt{LMO} and $\mathbb{E}_G[\cdot]$ denotes the expectation over random Gaussian samples.

Finally, we compute an upper bound on the complexity of each iteration, i.e., inner iteration complexity, of Algorithm~\ref{Algo1}.

\paragraph{Upper bound on inner iteration complexity.} At each iteration $t$, Algorithm~\ref{Algo1} solves the subproblem \texttt{LMO}, which generates a unit vector $h$, such that
\begin{equation}\label{eqn:LMOProofConv}
\alpha\langle hh^T, \nabla g_t\rangle \geq \max_{d\in \mathcal{S}} \alpha\langle d, \nabla g_t\rangle - \frac{1}{2}\eta \gamma_t C_g^u,
\end{equation}
where $\gamma_t = \frac{2}{t+2}$, $\nabla g_t = \nabla g(X_t)$ and $\mathcal{S} = \{X\succeq 0:\textup{Tr}(X) \leq n\}$. Note that this problem is equivalent to approximately computing maximum eigenvector of the matrix $\nabla g_t$ which can be done using Lanczos algorithm~\cite{kuczynski1992estimating}.

\begin{lemma}[Convergence of Lanczos algorithm]\label{lemma:ConvLanczos}
Let $\rho \in (0,1]$ and $p \in (0,1/2]$. For $\nabla g_t\in \mathbb{S}_n$, the Lanczos method~\cite{kuczynski1992estimating}, computes a vector $h\in \mathbb{R}^n$, that satisfies
\begin{equation}\label{eqn:LancProofConv}
h^T \nabla g_t h\geq \lambda_{\textup{max}}(\nabla g_t) - \frac{\rho}{8}\|\nabla g_t\|
\end{equation}
with probability at least $1-2p$, after at most $q \geq \frac{1}{2} + \frac{1}{\sqrt{\rho}}\log(n/p^2)$ iterations.
\end{lemma}
This result is an adaptation of~\cite[Theorem~4.2]{kuczynski1992estimating} which provides convergence of Lanczos to approximately compute minimum eigenvalue and the corresponding eigenvector of a symmetric matrix. Let $N = \frac{1}{2}+\frac{1}{\sqrt{\rho}}\log(n/p^2)$. We now derive an upper bound on $N$. 

Comparing~\eqref{eqn:LancProofConv} and~\eqref{eqn:LMOProofConv}, we see that
\begin{equation}
\begin{split}
&\frac{1}{2}\eta\gamma_tC_g^u = \alpha\frac{\rho}{8}\|\nabla g_t\|\\
\Rightarrow &\frac{1}{\rho} = \frac{\alpha \|\nabla g_t\|}{4\eta\gamma_tC_g^u}
\end{split}
\end{equation}
Substituting the value of $\gamma_t$ in the equation above, and noting that $\gamma_t = \frac{2}{t+2} \geq \frac{2}{T+2}$, we have
\begin{equation}\label{eqn:ConvProof2}
\begin{split}
\frac{1}{\rho} = \frac{\alpha\|\nabla g_t\|(t+2)}{8\eta C_g^u} \leq \frac{\alpha\|\nabla g_t\|(T+2)}{8\eta C_g^u} = \frac{\alpha\|\nabla g_t\|(1+\eta)}{4\eta\epsilon\textup{Tr}(C)},
\end{split}
\end{equation}
where the last equality follows from substituting the value of $T$ (see~\eqref{eqn:BoundonT}).
We now derive an upper bound on $\|\nabla g_t\|$.

\begin{lemma}
Let $g(X) = \langle C,X\rangle - \beta \phi_M\left(\textup{diag}(X) - \mathbbm{1}, \left[-\frac{1}{k-1} - e_i^TXe_j\right]_{(i,j)\in E}\right)$, where $\phi_M(\cdot)$ is defined in~\eqref{eqn:penaltyPhiDef1}. We have $\|\nabla g_t \| \leq \textup{Tr}(C) (1+6(\sqrt{2|E|+n}))$.
\end{lemma}

\begin{proof}
For the function $g(X)$ as defined in the lemma, $\nabla g_t = C - \beta D$, where $D$ is matrix such that $D_{ii} \in [-1,1]$ for $i=1,\dotsc,n$, $D_{ij}\in [-1,1]$ for $(i,j)\in E$, and $D_{ij} = 0$ for $(i,j)\notin E$. Thus, we have
\begin{equation}\label{eqn:ConvProof1}
\max_k |\lambda_k(D)| \leq \sqrt{\textup{Tr}(D^TD)} = \sqrt{\sum_{i,j=1}^n |D_{ij}|^2} \leq \sqrt{2|E|+n},
\end{equation}
where the last inequality follows since there are at most $2|E|$ off-diagonal and $n$ diagonal nonzero entries in the matrix $D$ with each nonzero entry in the range $[-1,1]$. Now,
\begin{equation}
\begin{split}
\|\nabla g_t\| = \|C-\beta D\| &\underset{(i)}{\leq} \|C\| + \|-\beta D\| \\ &\leq \max_i|\lambda_i(C)|+ \max_i |\lambda_i (-\beta D)|\\ 
&\underset{(ii)}{\leq} \textup{Tr}(C) + \beta \sqrt{2|E|+n}\\
&\underset{(iii)}{\leq} \textup{Tr}(C) (1+6(\sqrt{2|E|+n})).
\end{split}
\end{equation}
where (i) follows from the triangle inequality for the spectral norm of $C-\beta D$, (ii) follows from~\eqref{eqn:ConvProof1} and  since $C$ is graph Laplacian and a positive semidefinite matrix, and (iii) follows by substituting $\beta = 6\textup{Tr}(C)$ as given in Lemma~\ref{lemma:SubOptFeasMkC}.
\end{proof}

Substituting $\alpha = n$, and the bound on $\|\nabla g_t\|$ in~\eqref{eqn:ConvProof2}, we have
\begin{gather}
\frac{1}{\rho} \leq \frac{1+\eta}{4\eta} \frac{n(1+6(\sqrt{2|E|+n}))}{\epsilon},\quad\textup{and}\\
N = \frac{1}{2}+\frac{1}{\sqrt{\rho}}\log(n/p^2) \leq \frac{1}{2} + \sqrt{\frac{1+\eta}{4\eta}}\sqrt{\frac{n(1+6(\sqrt{2|E|+n}))}{\epsilon}}\log(n/p^2) = N^u.
\end{gather}
\begin{equation}
\end{equation}
Finally, each iteration of Lanczos method performs a matrix-vector multiplication with $\nabla g_t$, which has at most $2|E|+n$ number of nonzero iterations, and $\mathcal{O}(n)$ additional arithmetic operations. Thus, the computational complexity of Lanczos method is  $\mathcal{O}(N^u(|E|+n))$. Moreover, Algorithm~\ref{Algo1} performs $\mathcal{O}(|E|+n)$ additional arithmetic operations so that the total inner iteration complexity is $\mathcal{O}(N^u(|E|+n))$, which can be written as $\mathcal{O}\left(\frac{\sqrt{n}|E|^{1.25}}{\sqrt{\epsilon}}\log(n/p^2)\right)$. 

\paragraph{Computational complexity of Algorithm~\ref{Algo1}.} Now, substituting $p = \frac{\epsilon}{T(n,\epsilon)}$, we have
\begin{equation}
\log \left(\frac{n}{p^2}\right) = \log \left(\frac{(144)^2n^5(\log(2n+|E|))^2}{\epsilon^6}\right) \leq \log\left(\frac{(5.3n)^6}{\epsilon^6} \right) = 6\log\left(\frac{5.3n}{\epsilon}\right),
\end{equation}
where the inequality follows since $|E| \leq \binom{n-1}{2}$, $\left(\log\left(2n+\binom{n-1}{2}\right)\right)^2 \leq n$ for $n \geq 1$ and $(5.3)^6 \geq (144)^2$.
Substituting the upper bound on $\log (n/p^2)$ in $N^u$, and combining the inner iteration complexity, $\mathcal{O}(N^u(|E|+n))$, and outer iteration complexity, $T$, we see that Algorithm~\ref{Algo1} is a $\mathcal{O}\left(\frac{n^{2.5}|E|^{1.25}}{\epsilon^{2.5}}\log(n/\epsilon)\log(|E|)\right)$-time algorithm.
\end{proof}

\subsection{Proof of Lemma~\ref{lemma:SubFeasCC}}\label{appendix:ProofCC}
\begin{proof}
We need to prove four inequalities given in Lemma~\ref{lemma:SubFeasCC}.

\paragraph{Lower bound on the objective, $\langle C, \widehat{X}_{\epsilon} \rangle$.} Substituting the values of $\beta$ and $M$, and replacing $\epsilon$ by $\epsilon \textup{Tr}(C)$ in~\eqref{eqn:SubOptimalSolPenaltyProb2}, we have
\begin{equation}\label{eqn:SubOptimalSolPenaltyProb2CC}
\langle C, \widehat{X}_{\epsilon}\rangle \geq \langle C, X^{\star}_G\rangle - 2\epsilon \left(\textup{Tr}(L_{G^-})+\sum_{ij\in E^+}w^+_{ij}\right).
\end{equation}
Since $0.5I+0.5\mathbbm{1}\mathbbm{1}^T$ is feasible for~\eqref{prob:CorrCluster1}, $0.5(\textup{Tr}(L_{G^-})+\sum_{ij\in E^+}w^+_{ij}) \leq \langle C, X^{\star}_G\rangle$. Combining this fact with~\eqref{eqn:SubOptimalSolPenaltyProb2CC}, we have
\begin{equation*}
\langle C, \widehat{X}_{\epsilon}\rangle \geq (1-4\epsilon)\langle C, X^{\star}_G\rangle.
\end{equation*}

\paragraph{Bound on infeasibility.}
Let $E = E^- \cup E^+$ and let $\nu= [\nu^{(1)},\nu^{(2)}]\in \mathbb{R}^{n+|E|}$ be the dual variable such that $\nu^{(1)}$ is the dual variable corresponding to the $n$ equality constraints and $\nu^{(2)}$ is the dual variable for $|E|$ inequality constraints.
Following~\eqref{prob:SDP-Dual}, the dual of~\eqref{prob:CorrCluster1} is

\begin{equation}\label{prob:CC-Dual}
\tag{Dual-CC}
\min_{\nu}\quad \sum_{i=1}^n \nu^{(1)}_{i}\quad\textup{subject to}\quad \begin{cases}&\textup{diag}^*(\nu^{(1)}) + \sum\limits_{\substack{ij \in E\\i<j}} [e_ie_j^T+e_je_i^T]\frac{\nu^{(2)}_{ij}}{2} - C \succeq 0\\
&\nu^{(2)} \leq 0,\end{cases}
\end{equation}
where $C = L_{G^-}+W^+$. Let $\nu^{\star}$ be an optimal dual solution.
We derive an upper bound on $\|\nu^{\star}\|_1$ in the following lemma, which is then used to bound the infeasibility using~\eqref{eqn:BoundonInfeas}.

\begin{lemma}
The value of $\|\nu^{\star}\|_1$ is upper bounded by $2\left(\textup{Tr}(L_{G^-})+\sum_{ij\in E^+}w^+_{ij}\right)$.
\end{lemma}

\begin{proof}
For~\eqref{prob:CC-Dual}, $\nu^{(1)}_{i}=[L_{G^-}]_{ii} + \sum_{j:ij\in E^+} w_{ij}^+$ for $i=1,\dotsc,n$, and $\nu^{(2)}_{ij} = 2[L_{G^-}]_{ij}$ for $(i,j)\in E,i<j$ is a feasible solution. The optimal objective function value of~\eqref{prob:CC-Dual} is then upper bounded as
\begin{equation}\label{eqn:l1bound2}
\sum_{i=1}^{n}\nu^{(1)\star}_{i} \leq \textup{Tr}(L_{G^-})+\sum_{ij\in E^+}w^+_{ij}.
\end{equation}
We have $\left\langle \textup{diag}^*(\nu^{(1)\star}) + \sum\limits_{\substack{ij \in E\\i<j}} [e_ie_j^T+e_je_i^T]\frac{\nu^{(2)\star}_{ij}}{2} - C , \mathbbm{1}\mathbbm{1}^T\right\rangle \geq 0$ since both matrices are PSD. Using the fact that $\langle L_{G^-}, \mathbbm{1}\mathbbm{1}^T\rangle = 0$, and rearranging the terms, we have
\begin{equation*}
-\sum_{\substack{ij \in E\\i<j}}\nu^{(2)\star}_{ij} \leq \sum_{i=1}^{n}\nu^{(1)\star}_{i} - \sum_{ij\in E^+}w^+_{ij}.
\end{equation*}
Since $\nu^{(2)\star}\leq 0$, we can write
\begin{equation}\label{eqn:l1bound1}
\|\nu^{\star}\|_1 =  \sum_{i=1}^{n}|\nu^{(1)\star}_{i}| - \sum_{\substack{ij \in E\\i<j}}\nu^{(2)\star}_{ij} \leq 2\sum_{i=1}^{n}\nu^{(1)\star}_{i} - \sum_{ij\in E^+}w^+_{ij},
\end{equation}
where we have used the fact that for any dual feasible solution, $\nu^{(1)}_{i}\geq [L_{G^-}]_{ii}\geq 0$ for all $i=1,\dotsc,n$. Substituting~\eqref{eqn:l1bound2} in~\eqref{eqn:l1bound1},
\begin{equation}\label{eqn:l1bound3}
\|\nu^{\star}\|_1 \leq 2\textup{Tr}(L_{G^-})+\sum_{ij\in E^+}w^+_{ij} \leq 2\left(\textup{Tr}(L_{G^-})+\sum_{ij\in E^+}w^+_{ij}\right).
\end{equation}
\end{proof}

For $\Delta = \textup{Tr}(L_{G^-})+\sum_{ij\in E^+}w^+_{ij}$, $\widehat{X}_{\epsilon}$ is an $\epsilon\Delta$-optimal solution to~\eqref{prob:CC-LSE}. And so, we replace $\epsilon$ be $\epsilon\Delta$ in~\eqref{eqn:BoundonInfeas}. Now, substituting~\eqref{eqn:l1bound3} and the values of $\beta$ and $M$ into~\eqref{eqn:BoundonInfeas}, we get
\begin{equation}\label{eqn:BoundonInfeas1CC}
\max \left\{ \|\textup{diag}(\widehat{X}_{\epsilon}) - \mathbbm{1}\|_{\infty}, \max_{ij\in E,i<j}  - [\widehat{X}_{\epsilon}]_{ij} \right\} \leq \epsilon.
\end{equation}
This condition can also be stated as
\begin{equation*}
\|\textup{diag}(\widehat{X}_{\epsilon}) - \mathbbm{1}\|_{\infty} \leq \epsilon,\quad [\widehat{X}_{\epsilon}]_{ij} \geq - \epsilon\quad (i,j)\in E, i<j.
\end{equation*}

Substituting~\eqref{eqn:BoundonInfeas1CC}, ~\eqref{eqn:l1bound3} and the values of the parameters $\beta$ and $M$ into~\eqref{eqn:upperboundSDP} gives
\begin{equation*}
\langle C, \widehat{X}_{\epsilon} \rangle \leq \langle C, X^{\star}_G\rangle + 2\left(\textup{Tr}(L_{G^-})+\sum_{ij\in E^+}w^+_{ij}\right)\epsilon \leq (1+4\epsilon)\langle C, X^{\star}_G\rangle,
\end{equation*}
where the last inequality follows since $I$ is a feasible solution to~\eqref{prob:CorrCluster1}.
\end{proof}

\subsection{Proof of Lemma~\ref{lemma:CC-ApproxSol}}
\begin{proof}
We first note that Algorithm~\ref{algo:GenkSamples} generates a samples whose covariance is feasible to~\eqref{prob:CorrCluster1}.
\begin{prop}\label{prop:AlgoGenSamplesProof2}
Let $z_1,z_2\sim \mathcal{N}(0,\widehat{X}_{\epsilon})$ be Gaussian random vectors such that their covariance $\widehat{X}_{\epsilon}$ satisfies the inequality~\eqref{eqn:FeasBoundCC}. Replace Step 3 of Algorithm~\ref{algo:GenkSamples} with redefined $\textup{err}$ such that $\textup{err} = \max \{ 0, \max_{(i,j)\in E,i<j}\ - [\widehat{X}_{\epsilon}]_{ij}\}$. The Gaussian random vectors $z^f_1,z^f_2\sim \mathcal{N}(0,X^f)$ generated by the modified Algorithm~\ref{algo:GenkSamples} have covariance that is feasible to~\eqref{prob:CorrCluster1}.
\end{prop}

The proof of Proposition~\ref{prop:AlgoGenSamplesProof2} is the same as the proof of Proposition~\ref{prop:AlgoGenSamplesProof}. Now, let
\begin{equation}
X^f = \frac{\widehat{X}_{\epsilon}+\textup{err} \mathbbm{1}\mathbbm{1}^T}{\max (\textup{diag}(\widehat{X}_{\epsilon}))+\textup{err}} + \left(I - \textup{diag}^*\left(\frac{\textup{diag}(\widehat{X}_{\epsilon})+\textup{err}}{\max (\textup{diag}(\widehat{X}_{\epsilon}))+\textup{err}}\right)\right)
\end{equation}
The objective function value of~\eqref{prob:CorrCluster1} at $X^f$ is
\begin{align}
\langle C, X^f\rangle &= \left\langle C, \frac{\widehat{X}_{\epsilon}+\textup{err} \mathbbm{1}\mathbbm{1}^T}{\max (\textup{diag}(\widehat{X}_{\epsilon}))+\textup{err}} + \left(I - \textup{diag}^*\left(\frac{\textup{diag}(\widehat{X}_{\epsilon})+\textup{err}}{\max (\textup{diag}(\widehat{X}_{\epsilon}))+\textup{err}}\right)\right)\right\rangle\\
&\underset{(i)}{\geq}\ \frac{\langle C, \widehat{X}_{\epsilon}\rangle}{\max (\textup{diag}(\widehat{X}_{\epsilon}))+\textup{err}} + \left\langle C, \left(I - \textup{diag}^*\left(\frac{\textup{diag}(\widehat{X}_{\epsilon})+\textup{err}}{\max (\textup{diag}(\widehat{X}_{\epsilon}))+\textup{err}}\right)\right)\right\rangle\\
&\underset{(ii)}{\geq}\ \frac{\langle C, \widehat{X}_{\epsilon}\rangle}{\max (\textup{diag}(\widehat{X}_{\epsilon}))+\textup{err}}\ \underset{(iii)}{\geq}\  \frac{1-4\epsilon}{1+2\epsilon}\langle C, X^{\star}_G\rangle\ \underset{(iv)}{\geq}\ (1-6\epsilon)\langle C, X^{\star}_G\rangle
\label{eqn:CCCutboundproof1}
\end{align}
where (i) follows from the fact that $\langle L_{G^-}, \textup{err} \mathbbm{1}\mathbbm{1}^T\rangle = 0$ and $\langle W_{+},\textup{err} \mathbbm{1}\mathbbm{1}^T\rangle \geq 0$, (ii) follows since $L_{G^-}$ and $I -  \textup{diag}^*\left(\frac{\textup{diag}(\widehat{X}_{\epsilon})+\textup{err}}{\max (\textup{diag}(\widehat{X}_{\epsilon}))+\textup{err}}\right)$ are PSD and their inner product is nonnegative and the diagonal entries of $W_{+}$ are 0, (iii) follows from Lemma~\ref{lemma:SubFeasCC} and the fact that $\textup{err} \leq \epsilon$, and (iv) follows since $1-4\epsilon\geq (1+2\epsilon)(1-6\epsilon)$.
Combining the fact that $\langle C, X^{\star}_G\rangle \geq \textup{opt}_{CC}^G$ and $\mathbb{E}[\mathcal{C}]\geq 0.766\langle C, X^f\rangle$ with the above, we have
\begin{equation}
\mathbb{E}[\mathcal{C}] \geq 0.766(1-6\epsilon)\textup{opt}_{CC}^G.
\end{equation}
\end{proof}

\subsection{Proof of Lemma~\ref{lemma:ConvCC}}
\begin{proof}
We use Algorithm~\ref{Algo1} with $p = \frac{\epsilon}{T(n,\epsilon)}$ and $T(n,\epsilon) = \frac{64\log(2n+|E|)n^2}{\epsilon^2}$ to generate an $\epsilon\Delta$-optimal solution to~\eqref{prob:CC-LSE}, where $\Delta = \textup{Tr}(L_{G^-})+\sum_{ij\in E^+}w^+_{ij}$. 

\paragraph{Upper bound on outer iteration complexity.}
The convergence result given in Theorem~\ref{thm:FWapproxConv} holds when Algorithm~\ref{Algo1} is applied to~\eqref{prob:maxkCutLSE}.
Then, the total number of iterations of Algorithm~\ref{Algo1}, also known as outer iteration complexity, required to generate $\epsilon\Delta$-optimal solution to~\eqref{prob:maxkCutLSE} is
\begin{equation}\label{eqn:BoundonTCC}
T = \frac{2C_g^u(1+\eta)}{\epsilon\Delta} - 2 = \frac{2\beta M n^2(1+\eta)}{\epsilon\Delta} - 2.
\end{equation}

\paragraph{Bound on the value of generated clustering.} Algorithm~\ref{Algo1} with $p = \frac{\epsilon}{T(n,\epsilon)}$ generates a solution $\widehat{X}_{\epsilon}$ that satisfies the bounds in Lemma~\ref{lemma:SubOptFeasMkC} with probability with at least $1-\epsilon$ after at most $T$ iterations. Thus, the bound given in~\eqref{eqn:CCCutboundproof1} holds with probability at least $1-\epsilon$ and we have
\begin{equation}
\mathbb{E}[\langle C, X^f\rangle] \geq (1-6\epsilon)\langle C, X^{\star}_G\rangle (1-\epsilon) \geq (1-7\epsilon)\langle C, X^{\star}_{G}\rangle.
\end{equation}
Let $\mathbb{E}_L[\cdot]$ denote the expectation over the randomness in the subproblem \texttt{LMO} and $\mathbb{E}_G[\cdot]$ denote the expectation over random Gaussian samples.
The expected value of the generated clustering is then bounded as
\begin{equation}
\mathbb{E}[\mathcal{C}] = \mathbb{E}_L[\mathbb{E}_G[\mathcal{C}]] \underset{(i)}{\geq} 0.766 \mathbb{E}_L[\langle C, X^f\rangle] \geq 0.766 (1-7\epsilon)\langle C, X^{\star}_{G}\rangle\geq 0.766 (1-7\epsilon)\textup{opt}_{CC}^G,
\end{equation}
where (i) follows from the fact that the value of clustering generated by CGW rounding scheme satisfies $\mathbb{E}[\mathcal{C}] \geq 0.766\langle C, X^f\rangle$.

We now determine the inner iteration complexity of Algorithm~\ref{Algo1}.

\paragraph{Upper bound on inner iteration complexity.}
At each iteration $t$ of Algorithm~\ref{Algo1}, the subroutine \texttt{LMO} (see~\eqref{eqn:LMOProofConv}) is equivalent to approximately computing maximum eigenvector of the matrix $\nabla g_t$. This is achieved using Lanczos method whose convergence is given in Lemma~\ref{lemma:ConvLanczos}. Now, let $N = \frac{1}{2}+\frac{1}{\rho}\log(n/p^2)$. We see that the bound on $1/\rho$ is 
\begin{equation}\label{eqn:ConvProof3}
\frac{1}{\rho}  \leq \frac{\alpha\|\nabla g_t\|(1+\eta)}{4\eta\epsilon\Delta},
\end{equation}
which is similar to~\eqref{eqn:ConvProof2}. We now derive an upper bound on $\|\nabla g_t\|$.

\begin{lemma}
Let $g(X) = \langle L_{G^-} + W^+,X\rangle - \beta \phi_M\left(\textup{diag}(X) - \mathbbm{1}, \left[- e_i^TXe_j\right]_{(i,j)\in E}\right)$, where $\phi_M(\cdot)$ is defined in~\eqref{eqn:penaltyPhiDef1}. We have $\|\nabla g_t \| \leq \Delta (1+4(\sqrt{2|E|+n}))$, where $\Delta = \textup{Tr}(L_{G^-})+\sum_{ij\in E^+}w^+_{ij}$.
\end{lemma}

\begin{proof}
For the function $g(X)$ as defined in the lemma, $\nabla g_t = L_{G^-} + W^+ - \beta D$, where $D$ is matrix such that $D_{ii} \in [-1,1]$ for $i=1,\dotsc,n$, $D_{ij}\in [-1,1]$ for $(i,j)\in E$, and $D_{ij} = 0$ for $(i,j)\notin E$, and $E = E^- \cup E^+$. 
We have
\begin{gather}
\label{eqn:ConvProof4}
\max_k |\lambda_k(W^+)| \leq \sqrt{\textup{Tr}(W^{+T}W^+)} = \sqrt{\sum_{(i,j)\in E^+} |w^+_{ij}|^2} \leq \sum_{(i,j)\in E^+} w^+_{ij}, \quad \textup{and}\\
\label{eqn:ConvProof5}
\max_k |\lambda_k(D)| \leq \sqrt{\textup{Tr}(D^TD)} = \sqrt{\sum_{i,j=1}^n |D_{ij}|^2} \leq \sqrt{2|E|+n},
\end{gather}
\begin{equation}
\end{equation}
where the last inequality follows since $D$ has at most $2|E|+n$ nonzero entries in the range $[-1,1]$. Now, we have
\begin{equation}
\begin{split}
\|\nabla g_t\| = \|L_{G^-}+W^+-\beta D\|&\underset{(i)}{\leq} \|L_{G^-}\| + \|W^+\| + \|-\beta D\|\\
&\leq \max_i|\lambda_i(L_{G^-})|+ \max_i|\lambda_i(W^+)| +  \max_i |\lambda_i (-\beta D)|\\ 
&\underset{(ii)}{\leq} \textup{Tr}(L_{G^-}) + \sum_{(i,j)\in E^+} w^+_{ij} + \beta \sqrt{2|E|+n}\\
&\underset{(iii)}{\leq} \Delta (1+4(\sqrt{2|E|+n})).
\end{split}
\end{equation}
where (i) follows since the spectral norm of $L_{G^-}+W^+-\beta D$ satisfies the triangle inequality, (ii) follows from~\eqref{eqn:ConvProof4}, ~\eqref{eqn:ConvProof5} and the fact that $L_{G^-}$ is a positive semidefinite matrix, and (iii) follows by substituting the value of $\Delta$ and $\beta = 4\Delta$ as given in Lemma~\ref{lemma:SubFeasCC}.
\end{proof}
Substituting the bound on $\|\nabla g_t\|$ in~\eqref{eqn:ConvProof3}, we have
\begin{gather}
\frac{1}{\rho} \leq \frac{1+\eta}{4\eta} \frac{n(1+4(\sqrt{2|E|+n}))}{\epsilon},\quad\textup{and}\\
N = \frac{1}{2}+\frac{1}{\sqrt{\rho}}\log(n/p^2) \leq \frac{1}{2} + \sqrt{\frac{1+\eta}{4\eta}}\sqrt{\frac{n(1+4(\sqrt{2|E|+n}))}{\epsilon}}\log(n/p^2) = N^u.
\end{gather}
\begin{equation}
\end{equation}
The computational complexity of Lanczos method is $\mathcal{O}(N^u(|E|+n))$, where the term $|E|+n$ appears since Lanczos method performs matrix-vector multiplication with $\|\nabla g_t\|$, whose sparsity is $\mathcal{O}(|E|)$, plus additional $\mathcal{O}(n)$ arithmetic operations at each iteration. We finally write the computational complexity of each iteration of Algorithm~\ref{Algo1} as $\mathcal{O}\left(\frac{\sqrt{n}|E|^{1.25}}{\sqrt{\epsilon}}\log(n/p^2)\right)$. 

\paragraph{Total computational complexity of Algorithm~\ref{Algo1}.} Since $p = \frac{\epsilon}{T(n,\epsilon)}$, we have
\begin{equation}
\log \left(\frac{n}{p^2}\right) = \log \left(\frac{(64)^2n^5(\log(2n+|E|))^2}{\epsilon^6}\right) \leq \log\left(\frac{4^6n^6}{\epsilon^6} \right) = 6\log\left(\frac{4n}{\epsilon}\right),
\end{equation}
where the inequality follows since $|E| \leq \binom{n-1}{2}$ and $\left(\log\left(2n+\binom{n-1}{2}\right)\right)^2 \leq n$ for $n \geq 1$.
Multiplying outer and inner iteration complexity and substituting the bound on $p$, we prove that Algorithm~\ref{Algo1} is a $\mathcal{O}\left(\frac{n^{2.5}|E|^{1.25}}{\epsilon^{2.5}}\log(n/\epsilon)\log(|E|)\right)$-time algorithm.
\end{proof}

\subsection{Proof of Lemma~\ref{lemma:ApproxCutMkCSparse}}
For any symmetric matrix $X\in \mathbb{S}^n$, the definition of $\tau$-spectral closeness (Definition~\ref{def:specProp}) implies
\begin{equation}\label{eqn:SpecPropMatrix}
(1-\tau) \langle L_G, X\rangle \leq \langle L_{\Tilde{G}}, X\rangle \leq (1+\tau) \langle L_G, X\rangle.
\end{equation}
Let $C$ and $\Tilde{C}$ be the cost matrix in the objective of~\eqref{prob:maxkcutApprox}, when the problem is defined on the graphs $G$ and $\Tilde{G}$ respectively. Since $C$ and $\Tilde{C}$ are scaled Laplacian matrices (with the same scaling factor $(k-1)/2k$, from~\eqref{eqn:SpecPropMatrix}, we can write
\begin{equation}\label{eqn:ProofSparseMkC}
(1-\tau) \langle C, X\rangle \leq \langle \Tilde{C}, X\rangle \leq (1+\tau) \langle C, X\rangle.
\end{equation}
Let $X_G^{\star}$ and $X_{\Tilde{G}}^{\star}$ be optimal solutions to~\eqref{prob:maxkcutApprox} defined on the graphs $G$ and $\Tilde{G}$ respectively. From~\eqref{eqn:ProofSparseMkC}, we can write,
\begin{equation}\label{eqn:ProofSparseMkC1}
(1-\tau)\langle C, X^{\star}_G\rangle \leq \langle \Tilde{C}, X^{\star}_G\rangle \leq \langle \Tilde{C}, X^{\star}_{\Tilde{G}}\rangle,
\end{equation}
where the last inequality follows since $X_G^{\star}$ and $X_{\Tilde{G}}^{\star}$ are feasible and optimal solutions respectively to~\eqref{prob:maxkcutApprox} defined on the graph $\Tilde{G}$. Combining this with the bound in Lemma~\ref{lemma:MaxkCutapproxSol}, i.e., $\mathbb{E}[\texttt{CUT}] \geq \alpha_k (1-4\epsilon)\langle \Tilde{C}, X_{\Tilde{G}}^{\star}\rangle$, we get
\begin{equation}
\begin{split}
\mathbb{E}[\texttt{CUT}] \geq \alpha_k (1-4\epsilon)\langle \Tilde{C}, X_{\Tilde{G}}^{\star}\rangle \underset{(i)}{\geq} \alpha_k (1-4\epsilon) (1-\tau)\langle C, X^{\star}_G\rangle &\underset{(ii)}{\geq} \alpha_k(1-4\epsilon - \tau) \langle C, X^{\star}_G\rangle\\
&\underset{(iii)}{\geq} \alpha_k (1-4\epsilon - \tau) \textup{opt}_k^G,
\end{split}
\end{equation}
where (i) follows from~\eqref{eqn:ProofSparseMkC1}, (ii) follows since $(1-4\epsilon)(1-\tau) = 1 - 4\epsilon - \tau + 4\epsilon\tau \geq 1-4\epsilon\tau$ for nonnegative $\epsilon$ and $\tau$, and (iii) follows since $\langle C, X^{\star}_G\rangle \geq \textup{opt}_k^G$ for an optimal solution $X^{\star}_G$ to~\eqref{prob:maxkcutApprox} defined on the graph $G$.

\subsection{Proof of Lemma~\ref{lemma:ApproxCutCCSparse}}\label{sec:ProofCCSparseCut}
\begin{proof}
The Laplacian matrices $L_{G^-}$ and $L_{\Tilde{G}^-}$ of the graphs $G^-$ and its sparse approximation $\Tilde{G}^-$ respectively satisfy~\eqref{eqn:SpecPropMatrix}. Furthermore, let $L_{G}^+ = D^+ - W^+$, where $D^+_{ii} = \sum_{j:(i,j)\in E^+}w^+_{ij}$, be the Laplacian of the graph $G^+$ and similarly let $L_{\Tilde{G}}^+ = \Tilde{D}^+ - \Tilde{W}^+$ be the Laplacian of the graph $\Tilde{G}^+$. If $X=I$, from~\eqref{eqn:SpecPropMatrix}, we have
\begin{equation}\label{eqn:proofCCsparse1}
(1-\tau)\textup{Tr}(D^+)\leq \textup{Tr}(\Tilde{D}^+)\leq (1+\tau)\textup{Tr}(D^+).
\end{equation}
Rewriting the second inequality in~\eqref{eqn:SpecPropMatrix} for $X=X^{\star}_G$, and noting that $\textup{diag}(X^{\star}_G) = \mathbbm{1}$, we have
\begin{equation}\label{eqn:SpecPropW}
\begin{split}
\langle W^+,X^{\star}_G\rangle &\leq \frac{\langle \Tilde{W}^+, X^{\star}_G\rangle}{1+\tau}  +\frac{(1+\tau)\textup{Tr}(D^+)-\textup{Tr}(\Tilde{D}^+)}{1+\tau}\\
& \leq \frac{\langle \Tilde{W}^+, X^{\star}_G\rangle}{1+\tau} +\frac{2\tau\textup{Tr}(D^+)}{1+\tau},
\end{split}
\end{equation}
where the second inequality follows from~\eqref{eqn:proofCCsparse1}. Let $C=L_{G^-}+W^+$ and $\Tilde{C} =L_{\Tilde{G}^-}+\Tilde{W}^+$ represent the cost in~\eqref{prob:CorrCluster1} and~\eqref{prob:CC-Sparse} respectively. Let $X^{\star}_G$ be an optimal solution to~\eqref{prob:CorrCluster1}. The optimal objective function value of~\eqref{prob:CorrCluster1} at $X^{\star}_G$ is $\langle C, X^{\star}_G\rangle$ and
\begin{equation}
\begin{split}
(1-\tau)\langle C, X^{\star}_G\rangle &= (1-\tau) \langle L_{G^-}, X^{\star}_G\rangle +(1-\tau) \langle W^+, X^{\star}_G\rangle \\
&\underset{(i)}{\leq}  \langle L_{\Tilde{G}^-}, X^{\star}_G\rangle + \frac{1-\tau}{1+\tau}\langle \Tilde{W}^+, X^{\star}_G\rangle + \frac{2\tau(1-\tau)}{1+\tau}\textup{Tr}(D^+)\\
&\underset{(ii)}{\leq} \langle \Tilde{C}, X^{\star}_G\rangle - \frac{2\tau}{1+\tau}\langle \Tilde{W}^+, X^{\star}_G\rangle + \frac{2\tau}{1+\tau}\textup{Tr}(\Tilde{D}^+)\\
&\underset{(iii)}{\leq} \langle \Tilde{C}, X^{\star}_{\Tilde{G}}\rangle + \frac{2\tau}{1+\tau}\langle \Tilde{C}, X^{\star}_{\Tilde{G}}\rangle,
\end{split}
\end{equation}
where (i) follows from~\eqref{eqn:SpecPropMatrix} and~\eqref{eqn:SpecPropW}, (ii) follows from~\eqref{eqn:proofCCsparse1}, and substituting $\Tilde{C} = L_{\Tilde{G}^-} + \Tilde{W}^+$ and rearranging the terms and (iii) holds true since $\langle \Tilde{W}^+, X^{\star}_G\rangle \geq 0$, and $I$ and $X^{\star}_G$ are feasible to~\eqref{prob:CC-Sparse} so that $\textup{Tr}(\Tilde{D}^+)\leq \langle \Tilde{C}, X^{\star}_{\Tilde{G}}\rangle$ and $\langle \Tilde{C}, X^{\star}_G\rangle\leq \langle \Tilde{C}, X^{\star}_{\Tilde{G}}\rangle$. Rearraning the terms, we have
\begin{equation}\label{eqn:proofCCSparse2}
\langle C, X^{\star}_G\rangle \leq \frac{1+3\tau}{1-\tau^2}\langle \Tilde{C}, X^{\star}_{\Tilde{G}}\rangle.
\end{equation}
Combining~\eqref{eqn:proofCCSparse2} with the fact that the expected value of clustering $\mathbb{E}[\mathcal{C}]$ generated for the graph $\Tilde{G}$ satisfies~\eqref{eqn:CC-ApproxCut}, we have
\begin{equation}
\begin{split}
\mathbb{E}[\mathcal{C}] \geq 0.766(1-6\epsilon)\langle \Tilde{C}, X^{\star}_{\Tilde{G}}\rangle \geq 0.766\frac{(1-6\epsilon)(1-\tau^2)}{1+3\tau}\langle C, X^{\star}_G\rangle\geq (1-6\epsilon-3\tau)(1-\tau^2)\textup{opt}_{CC}^G,
\end{split}
\end{equation}
where the last inequality follows since $(1-6\epsilon-3\tau)(1+3\tau)\leq 1-6\epsilon$.
\end{proof}

\subsection{Proof of Lemma~\ref{lemma:SummaryPropSparse}}\label{appendix:SparseSummaryProof}
The first step of the procedure given in Section~\ref{sec:sparsification} is to sparsify the input graph using the technique proposed in~\cite{kyng2017framework} whose computational complexity is $\mathcal{O}(|E|\log^2n)$. The second step when generating solutions to \textsc{Max-k-Cut} and \textsc{Max-Agree} is to apply the procedures given in Sections~\ref{sec:MkC} and~\ref{sec:CC} respectively. The computational complexity of this step is bounded as given in Propositions~\ref{lemma:ConvMkC} and~\ref{lemma:ConvCC} leading to a $\mathcal{O}\left(\frac{n^{2.5}|E|^{1.25}}{\epsilon^{2.5}}\log(n/\epsilon)\log(|E|)\right)$-time algorithm.

\paragraph{Bound on the value of generated $k$-cut.} Let $p = \frac{\epsilon}{T(n,\epsilon)}$ and $T(n,\epsilon) = \frac{144\log(2n+|E|)n^2}{\epsilon^2}$ as given in Lemma~\ref{lemma:ConvMkC}. Using the procedure given in Section~\ref{sec:MkC}, we have $\mathbb{E}[\texttt{CUT}] \geq \alpha_k (1-5\epsilon)\textup{opt}_k^{\Tilde{G}}$. From the proof of Lemma~\ref{lemma:ApproxCutMkCSparse}, we see that $\texttt{CUT}$ is then an approximate $k$-cut for the input graph $G$ such that $\mathbb{E}[\texttt{CUT}] \geq \alpha_k (1-5\epsilon-\tau)\textup{opt}_k^G$.

\paragraph{Bound on the value of generated clustering.} Let $p = \frac{\epsilon}{T(n,\epsilon)}$ and $T(n,\epsilon) = \frac{64\log(2n+|E|)n^2}{\epsilon^2}$ as given in Lemma~\ref{lemma:ConvCC} and let the procedure given in Section~\ref{sec:CC} be applied to the sparse graph $\Tilde{G}$. Then, the generated clustering satisfies $\mathbb{E}[\mathcal{C}] \geq 0.766(1-7\epsilon)\textup{opt}_{CC}^{\Tilde{G}}$. Combining this with the proof of Lemma~\ref{lemma:ApproxCutCCSparse}, we have $\mathbb{E}[\mathcal{C}] \geq 0.766 (1-7\epsilon-3\tau)(1-\tau^2)\textup{opt}_{CC}^G$.

\section{Preliminary Computational Results for \textsc{Max-k-Cut}}
We provide some preliminary computational results when generating an approximate $k$-cut on the graph $G$ using the approach outlined in Section~\ref{sec:MkC}. The aim of these experiments was to verify that the bounds given in Lemma~\ref{lemma:MaxkCutapproxSol} were satisfied in practice. First, we solved~\eqref{prob:maxkCutLSE} to $\epsilon\textup{Tr}(C)$-optimality using Algorithm~\ref{Algo1} with the input parameters set to $\alpha = n$, $\epsilon = 0.05$, $\beta = 6\textup{Tr}(C)$, $M = 6\frac{\log(2n)+|E|}{\epsilon}$. We then computed feasible samples using Algorithm~\ref{algo:GenkSamples} and then finally used the FJ rounding scheme on the generated samples. The computations were performed using MATLAB R2021a on a machine with 8GB RAM. The peak memory requirement was noted using the \texttt{profiler} command in MATLAB.

We performed computations on randomly selected graphs from \textsc{GSet} dataset. In each case, the infeasibility of the covariance of the generated samples was less than $\epsilon$, thus satisfying~\eqref{eqn:FeasBoundMkC}. The number of iterations of \textup{LMO} in Algorithm~\ref{Algo1} was also within the bounds given in Proposition~\ref{prop:MkCSummary}. To a generate $k$-cut, we generated 10 sets of $k$ i.i.d. zero-mean Gaussian samples with covariance $\widehat{X}_{\epsilon}$, and each set was then used to generate a $k$-cut for the input graph using FJ rounding scheme. Let $\texttt{CUT}_{\textup{best}}$ denote the value of the best $k$-cut amongst the 10 generated cuts. Table~\ref{table:maxkcutResults} shows the result for graphs from the \textrm{GSet} dataset with $k =3,4$. Note that, $\texttt{CUT}_{\textup{best}}\geq \mathbb{E}[\texttt{CUT}]\geq \alpha_k(1-4\epsilon)\langle C, X^{\star}\rangle \geq \alpha_k\frac{1-4\epsilon}{1+4\epsilon}\langle C,\widehat{X}_{\epsilon}\rangle$, where the last inequality follows from combining~\eqref{eqn:MkCoptimalCut} with~\eqref{eqn:OptBoundMkC}. Since we were able to generate the values, $\texttt{CUT}_{\textup{best}}$ and $\langle C,\widehat{X}_{\epsilon}\rangle$, we noted that the weaker bound $\texttt{CUT}_{\textup{best}}/\langle C,\widehat{X}_{\epsilon}\rangle = \textup{AR} \geq \alpha_k(1-4\epsilon)/(1+4\epsilon)$ was satisfied by every input graph when $\epsilon = 0.05$.   

Furthermore, Table~\ref{table:maxkcutResults} also shows that the memory used by our method was linear in the size of the input graph. To see this, consider the dataset G1, and note that for $k=3$, the memory used by our method was $1252.73 \textup{kB} \approx 8.02\times (|V|+|E|)\times 8$, where a factor of 8 denotes that MATLAB uses 8 bytes to store a real number. Similarly, for other instances in \textsc{GSet}, the memory used by our method to generate an approximate $k$-cut for $k=3,4$ was at most $c\times (|V|+|E|)\times 8$, where for each graph the value of $c$ was bounded by $c \leq 82$, showing linear dependence of the memory used on the size of the input graph.

{\footnotesize
\begin{center}
\begin{longtable}{|c|c|c| c| c| c| c| c| c|c|}
\caption{Result of generating a $k$-cut for graphs from \textsc{GSet} using the method outlined in Section~\ref{sec:MkC}. We have, $\textup{infeas} = \max \{ \| \textup{diag}(X) - 1\|_{\infty}, \max \{0, -[\widehat{X}_{\epsilon}]_{ij}-\frac{1}{k-1}\}\}$ and $\textup{AR} = \texttt{CUT}_{\textup{best}}/\langle C, \widehat{X}_{\epsilon}\rangle$.}\\
\hline
\parbox[c]{0.8cm}{\raggedright Dataset} & \parbox[c]{0.6cm}{\raggedright $|V|$} & \parbox[c]{0.8cm}{\raggedright $|E|$} & \parbox[c]{0.5cm}{\raggedright $k$} & \parbox[c]{1.2cm}{\raggedright \# Iterations ($\times 10^3$)} & \parbox[c]{0.7cm}{\raggedright$\textup{infeas}$} & \parbox[c]{1.0cm}{\raggedright $\langle C, \widehat{X}_{\epsilon}\rangle$} & \parbox[c]{0.8cm}{ $\textup{CUT}_{\textup{best}}$} & \parbox[c]{0.8cm}{\raggedright $\textup{AR}$} & \parbox[c]{1.2cm}{Memory required (in kB)}\label{table:maxkcutResults}\\
\hline
\endfirsthead
\multicolumn{10}{c}%
{\tablename\ \thetable\ -- \textit{Continued from previous page}} \\
\hline
\parbox[c]{0.8cm}{\raggedright Dataset} & \parbox[c]{0.6cm}{\raggedright $|V|$} & \parbox[c]{0.8cm}{\raggedright $|E|$} & \parbox[c]{0.5cm}{\raggedright $k$} & \parbox[c]{1.2cm}{\raggedright \# Iterations ($\times 10^3$)} & \parbox[c]{0.7cm}{\raggedright$\textup{infeas}$} & \parbox[c]{1.0cm}{\raggedright $\langle C, \widehat{X}_{\epsilon}\rangle$} & \parbox[c]{0.8cm}{ $\textup{CUT}_{\textup{best}}$} & \parbox[c]{0.8cm}{\raggedright $\textup{AR}$} & \parbox[c]{1.2cm}{Memory required (in kB)}\\
\hline
\endhead
\hline \multicolumn{10}{c}{\textit{Continued on next page}} \\
\endfoot
\hline
\endlastfoot
G1&800&19176&3&823.94&$4\times 10^{-4}$&15631&14266&0.9127&1252.73\\
G1&800&19176&4&891.23&$4\times 10^{-4}$&17479&15746&0.9&1228.09\\
G2&800&19176&3&827.61&$6\times 10^{-5}$&15629&14332&0.917&1243.31\\
G2&800&19176&4&9268.42&$8\times 10^{-5}$&17474&15786&0.903&1231.07\\
G3&800&19176&3&1242.53&$7\times 10^{-5}$&15493&14912&0.916&1239.57\\
G3&800&19176&4&1341.37&$7\times 10^{-45}$&17301&15719&0.908&1240.17\\
G4&800&19176&3&812.8&$9\times 10^{-5}$&15660&14227&0.908&1230.59\\
G4&800&19176&4&9082.74&$10^{-4}$&17505&15748&0.899&1223.59\\
G5&800&19176&3&843.5&$10^{-4}$&15633&14341&0.917&1222.09\\
G5&800&19176&4&9294.32&$10^{-4}$&17470&15649&0.895&1227.9\\
G14&800&4694&3&1240.99&0.002&3917&2533&0.646&3502.64\\
G14&800&4694&4&3238.42&0.001&4467.9&3775&0.844&519.25\\
G15&800&4661&3&3400.17&0.001&4018.6&3385&0.842&612\\
G15&800&4661&4&1603.13&0.001&4475.8&3754&0.838&648.17\\
G16&800&4672&3&33216.68&0.001&4035.7&3422&0.847&561\\
G16&800&4672&4&3059.11&0.001&4437.5&3783&0.852&2800\\
G17&800&4667&3&3526.4&0.001&4031.5&3414&0.846&602.81\\
G17&800&4667&4&3400.01&0.001&4440&3733&0.84&693.6\\
G22&2000&19990&3&7402.59&$10^{-4}$&17840&11954&0.67&1340.34\\
G22&2000&19990&4&8103.83&$10^{-4}$&19582&16670&0.851&1341.67\\
G23&2000&19990&3&3597.39&$10^{-4}$&17938&15331&0.854&1360.09\\
G23&2000&19990&4&3588.04&$10^{-4}$&19697&16639&0.844&1317.09\\
G24&2000&19990&3&4304.48&$10^{-4}$&17913&15370&0.858&1341.96\\
G24&2000&19990&4&1994.26&$10^{-4}$&19738&16624&0.842&1321.59\\
G25&2000&19990&3&9774.03&$10^{-4}$&18186&15294&0.841&1311.54\\
G25&2000&19990&4&1540.14&$10^{-4}$&19778&16641&0.841&1330.95\\
G26&2000&19990&3&2069.65&$10^{-4}$&18012&15411&0.855&1321.92\\
G26&2000&19990&4&1841.06&$2\times 10^{-4}$&19735&16609&0.841&1331.53\\
G43&1000&9990&3&894.53&$10^{-4}$&9029&7785&0.862&661.09\\
G43&1000&9990&4&9709.68&$2\times 10^{-4}$&9925&8463&0.852&665.59\\
G44&1000&9990&3&721.64&$10^{-4}$&9059.5&7782&0.859&661.09\\
G44&1000&9990&4&9294.43&$10^{-4}$&9926.1&8448&0.851&765.37\\
G45&1000&9990&3&794.84&$10^{-4}$&9038.4&7773&0.86&661.09\\
G45&1000&9990&4&9503.74&$2\times 10^{-4}$&9929.7&8397&0.845&669\\
G46&1000&9990&3&703.4&$10^{-4}$&9068.5&7822&0.862&661.09\\
G46&1000&9990&4&9684.93&$4\times 10^{-4}$&9929.9&8333&0.839&657.09\\
G47&1000&9990&3&777.61&$10^{-4}$&9059.4&7825&0.863&679.89\\
G47&1000&9990&4&9789.55&$2\times 10^{-4}$&9930.8&8466&0.852&661.09\\
\end{longtable}
\end{center}}

\section{Additional Computational Results for Correlation Clustering}\label{appendix:AdditionalResultsCC}
We provide the computational result for the graphs from the \textsc{GSet} dataset (not included in the main article) here. We performed computations for graphs G1-G57 from \textsc{GSet} dataset. The instances for which we were able to generate an $\epsilon\Delta$-optimal solution to~\eqref{prob:CC-LSE} are given in Table~\ref{table:CCAddResults}, where the parameters, $\epsilon$ and $\Delta$, were set as given in Section~\ref{sec:ComputationalResults}. For the instances not in the table, we were not able to generate an $\epsilon\Delta$-optimal solution after 30 hours of runtime.

{\footnotesize \begin{center}
\begin{longtable}{|c|c|c| c| c| c| c| c| c|c|}
\caption{Result of generating a clustering of graphs from \textsc{GSet} using the method outlined in Section~\ref{sec:CC}. We have, $\textup{infeas} = \max \{ \| \textup{diag}(X) - 1\|_{\infty}, \max \{0, -[\widehat{X}_{\epsilon}]_{ij}\}\}$, $\textup{AR} = \mathcal{C}_{\textup{best}}/\langle C, \widehat{X}_{\epsilon}\rangle$ and $0.75(1-6\epsilon)/(1+4\epsilon) = 0.4375$ for $\epsilon = 0.05$.}\\
\hline
\parbox[c]{0.8cm}{\raggedright Dataset} & \parbox[c]{0.4cm}{\raggedright $|V|$} & \parbox[c]{0.7cm}{\raggedright $|E^+|$} & \parbox[c]{0.7cm}{\raggedright $|E^-|$} & \parbox[c]{1.2cm}{\raggedright \# Iterations ($\times 10^3$)} & \parbox[c]{0.4cm}{\raggedright$\textup{infeas}$} & \parbox[c]{1.0cm}{\raggedright $\langle C, \widehat{X}_{\epsilon}\rangle$} & \parbox[c]{0.8cm}{ $\mathcal{C}_{\textup{best}}$} & \parbox[c]{0.6cm}{\raggedright $\textup{AR}$} & \parbox[c]{1.2cm}{Memory required (in kB)}\label{table:CCAddResults}\\
\hline
\endfirsthead
\multicolumn{10}{c}%
{\tablename\ \thetable\ -- \textit{Continued from previous page}} \\
\hline
\parbox[c]{0.8cm}{\raggedright Dataset} & \parbox[c]{0.4cm}{\raggedright $|V|$} & \parbox[c]{0.7cm}{\raggedright $|E^+|$} & \parbox[c]{0.7cm}{\raggedright $|E^-|$} & \parbox[c]{1.2cm}{\raggedright \# Iterations ($\times 10^3$)} & \parbox[c]{0.4cm}{\raggedright$\textup{infeas}$} & \parbox[c]{1.0cm}{\raggedright $\langle C, \widehat{X}_{\epsilon}\rangle$} & \parbox[c]{0.8cm}{ $\mathcal{C}_{\textup{best}}$} & \parbox[c]{0.8cm}{\raggedright $\textup{AR}$} & \parbox[c]{1.2cm}{Memory required (in kB)}\\
\hline
\endhead
\hline \multicolumn{10}{c}{\textit{Continued on next page}} \\
\endfoot
\hline
\endlastfoot
G2&800&2501&16576&681.65&$8\times 10^{-4}$&848.92&643.13&0.757&1520.18\\
G3&800&2571&16498&677.56&$7\times 10^{-4}$&835.05&634.83&0.76&1529.59\\
G4&800&2457&16622&665.93&$6\times 10^{-4}$&852.18&647.37&0.76&1752\\
G5&800&2450&16623&646.4&$10^{-3}$&840.63&636.21&0.756&1535.92\\
G6&800&9665&9511&429.9&$3\times 10^{-4}$&25766&21302&0.826&1664\\
G7&800&9513&9663&423.58&$8\times 10^{-4}$&26001&20790&0.799&1535.06\\
G8&800&9503&9673&421.34&$6\times 10^{-4}$&26005&21080&0.81&4284\\
G9&800&9556&9620&426.4&$3\times 10^{-4}$&25903&21326&0.823&1812\\
G10&800&9508&9668&426.25&$3\times 10^{-4}$&25974&21412&0.824&1535.59\\
G12&800&798&802&393.69&$9\times 10^{-4}$&3023.4&2034&0.672&444.06\\
G13&800&817&783&416.29&$8\times 10^{-4}$&3001.1&2010&0.669&613.03\\
G15&800&3801&825&284.77&$10^{-3}$&529.83&401.19&0.757&460.17\\
G16&800&3886&749&228.12&$8\times 10^{-4}$&524.69&417.88&0.796&451.07\\
G17&800&3899&744&2448.633&$9\times 10^{-4}$&536.65&369.04&0.687&480.45\\
G18&800&2379&2315&1919.44&$2\times 10^{-3}$&7237.6&5074&0.701&434.67\\
G19&800&2274&2387&2653.79&$2\times 10^{-3}$&7274.2&5130&0.705&496\\
G20&800&2313&2359&1881.75&$2\times 10^{-3}$&7258.1&5186&0.714&406.09\\
G21&800&2300&2367&1884.97&$2\times 10^{-3}$&7281.3&5238&0.719&467.26\\
G23&2000&120&19855&550.77&$2\times 10^{-3}$&1802.2&1373.2&0.762&1651.54\\
G24&2000&96&19875&812.16&$10^{-3}$&1811.2&1384.6&0.764&1678.04\\
G25&2000&109&19872&1739.06&$6\times 10^{-4}$&1801.8&1398.1&0.776&1650.48\\
G26&2000&117&19855&1125.74&$10^{-3}$&1789.9&1356.9&0.758&1650.01\\
G27&2000&9974&10016&464.93&$5\times 10^{-4}$&30502&22010&0.721&1647.09\\
G28&2000&9943&10047&553.65&$4\times 10^{-4}$&30412&22196&0.729&1317.78\\
G29&2000&10035&9955&513.97&$2\times 10^{-4}$&30366&23060&0.759&1310.46\\
G30&2000&10045&9945&594.09&$3\times 10^{-4}$&30255&22550&0.745&1310.48\\
G31&2000&9955&10035&1036.9&$2\times 10^{-4}$&29965&22808&0.761&1305.05\\
G33&2000&1985&2015&403.75&$10^{-3}$&7442&4404&0.591&634.93\\
G34&2000&1976&2024&863.53&$4\times 10^{-4}$&7307.2&4760&0.651&574.12\\
G44&1000&229&9721&515.18&$10^{-3}$&810.82&616.61&0.76&655.09\\
G45&1000&218&9740&504.91&$10^{-3}$&812.21&615.84&0.758&660.51\\
G46&1000&237&9723&469.6&$10^{-3}$&818.39&623.95&0.762&655.09\\
G47&1000&230&9732&495.24&$9\times 10^{-4}$&819.63&621.65&0.758&648.32\\
G49&3000&0&6000&1002.59&0.003&599.64&456.48&0.761&733\\
G50&3000&0&6000&996.19&0.004&599.64&455.78&0.76&540.26\\
G52&1000&4750&1127&2041.8&0.001&684.1&441.02&0.644&757.59\\
G53&1000&4820&1061&785.33&$8\times 10^{-4}$&695.53&445.03&0.639&417.07\\
G54&1000&4795&1101&2899.99&$7\times 10^{-4}$&686.8&482.57&0.702&517.09\\
G56&5000&6222&6276&1340.35&0.004&22246&12788&0.574&1243.98\\
\end{longtable}
\end{center}}

\end{document}